\documentclass[11pt]{amsart}
\usepackage{geometry}                
\usepackage{graphicx}
\usepackage{amsrefs}
\usepackage{amssymb}
\usepackage{mathtools}
\usepackage[matrix,arrow,ps]{xy}
\usepackage{hyperref}
\usepackage[capitalize]{cleveref}
\usepackage{tikz}
\usepackage{tikz-cd}
\newcommand\mlnode[1]{\begin{tabular}{@{}c@{}}#1\end{tabular}}
\usepackage{faktor}
\usepackage{epstopdf}
\newcommand{\aut}{\mathfrak{aut}}
\newcommand{\hol}{\mathfrak{hol}}

\newcommand{\C}{\mathbb{C}}
\newcommand{\CN}{\mathbb{C}^N}
\newcommand{\CNp}{\mathbb{C}^{N^\prime}}

\newcommand{\R}{\mathbb{R}}

\newcommand{\N}{\mathbb{N}}
\newcommand{\Z}{\mathbb{Z}}

\newcommand{\dop}[1]{\frac{\partial}{\partial #1}}

\newcommand{\holmaps}{\mathcal{H}}

\newcommand{\todo}[1]{}

\newcommand{\Sphere}[1]{\mathbb S^{#1}}
\newcommand{\parcur}[1]{\mathfrak P^{#1}}
\newcommand{\tz}{\mathbf t_0}
\newcommand{\p}{\mathbf p}

\newlength{\extendaxesby}

\DeclareMathOperator{\id}{id}
\DeclareMathOperator{\rank}{rk}

\DeclareMathOperator{\real}{Re}

\DeclareMathOperator{\Aut}{Aut}

 \newtheorem{thm}{Theorem}
\newtheorem{theorem}[thm]{Theorem}

\newtheorem{lem}[thm]{Lemma}
\newtheorem{lemma}[thm]{Lemma}

\newtheorem{proposition}[thm]{Proposition}
\newtheorem{cor}[thm]{Corollary}

\theoremstyle{definition}

\newtheorem{definition}[thm]{Definition}
\newtheorem{exa}{Example}
\newtheorem{example}[exa]{Example}

\newtheorem{remark}[thm]{Remark}

\title[Local rigidity and higher order infinitesimal deformations]{Sufficient and necessary conditions \\ for local rigidity of CR mappings \\ and higher order infinitesimal deformations}

\begin{document}

\author{Giuseppe della Sala}
\address{Department of Mathematics, American University of Beirut}
\email{gd16@aub.edu.lb}
\author{Bernhard Lamel}
\address{Fakult\"at f\"ur Mathematik, Universit\"at Wien}
\email{bernhard.lamel@univie.ac.at}
\author{Michael Reiter}
\address{Fakult\"at f\"ur Mathematik, Universit\"at Wien}
\email{m.reiter@univie.ac.at}

\begin{abstract} In this paper we continue our study of local rigidity for maps of CR submanifolds of the complex space. We provide a linear sufficient condition for local rigidity of finitely nondegenerate maps between minimal CR manifolds. Furthermore, we show higher order infinitesimal conditions can be used to give a characterization of local rigidity. 
\end{abstract}

\maketitle

\section{Introduction}

Let $M\subset \CN$ and $M'\subset \CNp$ be real submanifolds, and consider 
the set $\mathcal H(M,M')$ of holomorphic mappings $H$ defined in a neighborhood of $M$ in $\C^N$ and valued in $\C^{N'}$ satisfying $H(M) \subset M'$. The (holomorphic) automorphism groups of $M$ and $M'$ act on $\mathcal H (M,M')$ in the natural way, leading
to an action of $G = \Aut(M) \times \Aut(M')$ defined for $(\varphi, \psi) \in G$ by 
\[ (\varphi, \psi) \cdot H = \psi \circ H \circ \varphi^{-1}. \] 
A particularly interesting question about this action is under which conditions 
the image of a map $H\colon M \to M'$ in the quotient 
space is {\em isolated}; in this case we say that the map $H$ is {\em locally rigid}. The same setup applies in 
the local setting  
to germs of manifolds $(M,p) \subset (\CN,p)$, $(M',p') \subset (\CNp,p')$, and their automorphism groups $\Aut(M,p)$ and
$\Aut(M',p')$, respectively. 
A typical 
example of maps which are totally rigid are maps of hyperquadrics of {\em positive} signature, 
as in the ``super-rigidity'' encountered by Baouendi, Ebenfelt and Huang \cite{BEH08}; while typical 
examples of maps which are not totally rigid are maps of spheres with big codimension, such as the so-called ``D'Angelo family'' (see the survey of Huang--Ji \cite{HJ07}) connecting the flat and the Whitney map $S^3\to S^7$, given by $(z,w) \mapsto (z, \sin(\theta) w, \cos(\theta)zw, \cos(\theta) w^2)$ for $\theta\in [0,\pi/2]$.  

We have started an  investigation of $\mathcal H ((M,p),(M',p'))$ for general germs
 $(M,p)$ and $(M',p')$  in previous articles \cite{dSLR15a,dSLR15b}.  Local rigidity, both in the global and local CR settings, is again going to be the focus of this paper; our inspiration here draws mainly from 
the analogous notion of stability of smooth maps of smooth manifolds as introduced by J. Mather in a series of papers starting with \cite{Mather68,Mather69} in the framework of singularity theory. Mather concluded that stability is equivalent to 
its linearized notion, called infinitesimal stability. In our previous work we have given linear obstructions to local rigidity similar to those of Mather. However, we also showed that the notions of {\em infinitesimal rigidity} and 
local rigidity  are {\em not equivalent} (see \cite{dSLR15a}) in the local CR setting. In order
to overcome this problem, the present paper introduces
 higher order infinitesimal deformations (for more details see \cref{section:introhigherorderinfdef} below),
 which will be shown to encode necessary and sufficient infinitesimal conditions for local rigidity. In addition
 to this, we improve the sufficient conditions given in \cite{dSLR15a, dSLR15b}.  Namely, our new approach (exploiting the analytic structure in a deeper way) allows us to avoid studying the topological properties of the group action of automorphisms and thus to prove our results under more general conditions. 

One of the main technical tools dealing with maps of real-analytic CR manifolds and infinitesimal deformations is provided by jet parametrization results. Such results were obtained for many classes of manifolds and maps in the literature (\cite{BER97,Zaitsev97, BER99, KZ05, BRWZ04, LM07, LMZ08, JL13a, JL13b}) and can be of a very technical nature. In our work we will define the notion of a class of maps which satisfy the jet parametrization property, which we will take  for granted in many results, and also  provide an example of an interesting large class of mappings for which it is satisfied (for more details see \cref{sec:jetparam}).

Our first main result is a sufficient condition for local rigidity in terms of infinitesimal deformations. Given $H \in \mathcal H(M,M')$ a holomorphic vector field  $V$ of the 
form 
\[ V = \sum_{j=1}^{N'} V_j (Z) \dop{Z_j'}\biggr|_{H(Z)} \] 
is called an \textit{infinitesimal deformation of $H$} if $\real V$ is tangent to $M'$ along $H(M)$, i.e. if for any 
point $p\in M$ and 
any real valued real-analytic function $\varrho (Z',\bar Z')$ defined near  $H(p)\in M'$ and vanishing on $M'$ satisfies 
\[ (\real V \varrho) (Z) := \sum_{j=1}^{N'} \real V_j (Z) \varrho_{Z_j'} (H(Z), \overline{H(Z)}) = 0, \quad Z\in M\cap U,  \]
for some open neighbourhood $U$ of $p$.  We denote the collection of all infinitesimal deformations of $H$ by $\hol(H)$. For the special case that we are dealing with the identity map in $\CN$ or $\CNp$, respectively, we use the classical notation $\hol(M):= \hol(\id_M)$ 
and $\hol(M'):=\hol(\id_{M'})$. It turns out that if 
 $\Aut(M)$ and $\Aut(M')$ are Lie groups, then $\hol(M)$ and $\hol(M')$ are their respective Lie algebras.
We also define $\aut(H) = H_*(\hol(M))+\hol(M')|_H$ to be the subspace of $\hol(H)$ of \textit{trivial infinitesimal deformations of $H$}; it can be identified with the Lie algebra of the orbit $G \cdot H$.

\begin{theorem}
	\label{t:rigidmain}
	Let $M \subset \C^N$ and $M' \subset \C^{N'}$ be real-analytic submanifolds. Assume $G$ is a Lie group and let $\mathcal F \subset \mathcal H (M,M')$ be a class of holomorphic mappings satisfying the jet parametrization property. If $H: M \rightarrow M'$,  $H \in \mathcal F$, satisfies $\hol(H) = \aut(H)$, then $H$ is locally rigid.
\end{theorem}

The analogous statement for germs of maps is also true,  with the same proof, and therefore 
improves the results in \cite{dSLR15b}. The proof is carried out through a careful analysis of an analytic set $A$, contained in an appropriate finite dimensional jet space, which parametrizes the maps in $\mathcal F$.

Comparing with \cite[Theorem 2]{dSLR15b} we have relaxed the assumptions on $M$ and $M'$ considerably. The stronger assumptions in our earlier paper allowed us to prove that the action of $G$ on the set of maps is free and proper, which was needed in the proof given in \cite{dSLR15b}, while in the approach of the present paper we are able to  avoid the reliance on the topological properties of $G$.

Also remark that \cref{t:rigidmain} provides a counterpart of \cite[Theorem 1]{Mather69} in the setting of CR manifolds. However, the infinitesimal stability is in fact a necessary and sufficient condition for stability (see \cite[Theorem 4.1]{Mather70}), while we stress again that the infinitesimal condition given in \cref{t:rigidmain} is only sufficient, but 
not necessary for local rigidity.

Our second main result adresses this problem and provides
 a characterization of local rigidity in terms of \textit{higher order infinitesimal deformations of $H$}. An infinitesimal deformation of order $k$ of $H$ is represented by a curve of maps in $\mathcal H(M,\CNp)$, which is tangent to $\mathcal H (M,M')$ up to order $k$ at $H$, or equivalently by holomorphic 
maps $F_1, \dots ,F_k \in \mathcal H(M, \CNp)$ with the property that for any real-valued real-analytic function
$\varrho$ defined in a neighborhood of a point in $M'$ and vanishing on $M'$ we have that 
\[ \varrho\left(H(Z) + \sum_{\ell=1}^k F_\ell(Z) t^\ell,\overline{H(Z)} + \sum_{\ell=1}^k \overline{F_\ell(Z)} t^\ell \right) = O (t^{k+1}), \quad Z\in M.  \] The set of infinitesimal deformations of order $k$ of $H$ is denoted by $\hol^k(H)$. The projection of $\hol^k(H)$ on the first $j$ components we denote by $\hol^k_j(H)$. The subset of $\hol^k(H)$ of the trivial $k$-th order infinitesimal deformations is denoted by $\aut^k(H)$ (for more details see \cref{section:introhigherorderinfdef}).

\begin{theorem}\label{th:morepreciselymain}
Let $M,M'$ and $\mathcal F$ be as in \cref{t:rigidmain}. Suppose that $G$ is a Lie group with finitely many connected components. Then for each $H_0\in \mathcal F$ 
there exists a neighbourhood ${\mathcal U}_{H_0}$ and a function $j\mapsto \ell(j)$ 
such that $H\in {\mathcal U}_{H_0}$ is locally rigid if and only if $\aut^j(H)=\hol^{\ell(j)}_j(H)$ for all $j\in \mathbb N$.
\end{theorem}

The proof is achieved by extending the jet parametrization results to higher order infinitesimal deformations and making use of tools from real-analytic geometry. In particular, it relies heavily on  Artin's approximation theorem, especially in the strong approximation form first given by Wavrik. Approximation theorems and their uses in CR geometry 
are an interesting topic in itself; for a very good survey on the state-of-the-art in this area, we refer 
the reader to Mir's survey \cite{Mir14}.

The paper is organized as follows. In \cref{sec:preliminaries} we establish some of the main notation and recall some basic facts about maps of CR manifolds and the action of CR automorphisms. In \cref{sec:jetparam} we discuss one of the main technical tools of the paper, that is, a jet parametrization result for certain classes of maps of CR manifolds. Such results are well established in the literature, so we only  give more detailed proofs for less standard aspects of our formulation of the results. 
\cref{sec:mainresults} is devoted to the proof of the main results. Finally, in \cref{sec:examples} we show how to apply the results in the paper to some concrete examples.
\\ 
\\ 
\section{Preliminaries}\label{sec:preliminaries}
Let $M\subset \mathbb C^{N}$, $M'\subset \mathbb C^{N'}$ be closed generic real submanifolds of class $C^\omega$ (see \cref{sec:CRgeometry} below).

\subsection{Spaces of maps and local rigidity}

 We will be interested in studying the set of real analytic CR maps from $M$ to $M'$, which will be denoted as $\mathcal H(M,M')$. Furthermore we denote by $\mathcal H(M,\C^{N'})$ the set of real analytic CR maps $M\to \C^{N'}$. Each element of $\holmaps = \mathcal H(M,\C^{N'})$ extends holomorphically to a neighborhood $U$ (depending on $H$) of $M$ in $\mathbb C^N$. In general, given a neighborhood $U$ of $M$ in $\mathbb C^N$, we denote by $\holmaps(U,\mathbb C^{N'})$ the space of holomorphic maps from $U$  to $\C^{N'}$. Thus, choosing a fundamental system of neighborhoods $U_n$ of $M$ in $\C^N$, we have $\mathcal H(M,\C^{N'})=\cup_{n\in \mathbb N} \mathcal H(U_n,\C^{N'})$; we usually endow  $\mathcal H(M,\C^{N'})$ with the natural inductive limit topology. If $M$ is compact, the resulting space is a (DFS) space.

We will denote by $\Aut(M)$ (respectively $\Aut(M')$) the group of CR automorphisms of $M$ ($M'$); furthermore we define $G=\Aut(M)\times\Aut(M')$. In the paper we will assume that $G$ is a Lie group (in particular, that it has finitely many connected components). The group $G$ acts on $\mathcal H(M, M')$ as follows: given $g=(\sigma, \sigma')\in G$ we define a map $\mathcal H(M, M') \to \mathcal H(M, M')$ by
\[\mathcal H(M, M') \ni H \to g\cdot H=\sigma' \circ H \circ \sigma^{-1} \in \mathcal H(M, M').\]
The set $\mathcal H(M,M')$ and the action of $G$ on it have been studied in several papers by many authors (\cite{Cartan32, CM74, Faran82, Webster79b, Huang99, Lebl11, HJ01, DAngelo88b, DAngelo91, EHZ04, EHZ05, BH05, DL16, DHX17, Reiter16b}), most notably in the case when $M$ and $M'$ are spheres. We recall from  \cite{dSLR15a, dSLR15b} the property of \emph{local rigidity},
 which roughly states that all the maps close enough to a given $ H\in \mathcal H(M,M')$ are equivalent to $H$ under the action of $G$.  

More precisely, the definition is as follows:
\begin{definition}\label{locrig}
Let $M\subset \mathbb C^{N}$, $M'\subset \mathbb C^{N'}$ be closed generic real submanifolds of class $C^\omega$, and let $H\in \mathcal H(M,M')$. 
We say that $H$ is \emph{locally rigid} if $H$ projects to an isolated point in the 
quotient $\faktor{\mathcal H(M,M')}{G}$.

Equivalently, $H$ is locally rigid  if and only if  there exists a neighborhood $\mathcal U$ of $H$ in $\mathcal H(M,\C^{N'})$ such that for every $\hat H\in \mathcal H(M,M')\cap \mathcal U$ there is $g\in G$ such that $\hat H=g\cdot H$ (cf.\ \cite{dSLR15a}, Remark 12).
\end{definition}

\subsection{Jet spaces}

In order to work in finite dimensional spaces, we will use jet parametrization results to identify maps with the collection of their derivatives at some point. To that aim it will be convenient to set up the appropriate notation.

We define  
\[J_0^k = \faktor{\mathbb C \{Z\}^{N'} }{\mathfrak{m}^{k+1}}, \]
where $\mathfrak{m}=(Z_1,\ldots,Z_N)$ is the maximal ideal, the space of the $k$-jets at $0$
of holomorphic maps $\mathbb C^N\to \mathbb C^{N'}$, with the 
natural projection $j_0^k$. Given $p\in \mathbb C^N$ it is straightforward to give an analogous definition for the space $J_p^k$ of jets  at $p$. For a given $k$, we will denote by $\Lambda$  the coordinates in $J_0^k$.

As in \cite{dSLR15a} it is possible in this setting to define an induced $\Aut_p(M) \times \Aut(M')$-action on the jet space $J_p^k$, where $\Aut_p(M)$ is the subgroup of elements of $\Aut(M)$ which fix $p \in M$ (i.e. $\Aut_p(M)$ is the stabilizer of $p \in M$).

To define an induced $G$-action on $J_p^k$ is decidedly more involved, and we will need to make use of the jet parametrization property;  we thus will only do this later in \cref{def:inducedAction}.

\subsection{Spaces of curves of maps}

Our main strategy for studying the set $\mathcal H(M,M')$ 
(and the action of $G$ on it) is to use appropriate parametrization results, which allow us 
to regard $\mathcal H(M,M')$ -- or a particular subset of it -- as a real-analytic subset $A$ of a finite dimensional space. The reduction to an analytic setting in turn leads us to the study of analytic curves on $A$, or tangent to it to high enough order, and allows to make use of the tools of real-analytic geometry. We will start by introducing convenient  notation for curves of maps.

For technical reasons we also need to introduce the following objects:

\begin{definition}
\label{def:PiTau}
For $H \in \mathcal H(M,M')$ and $\mathcal F$ a suitable space of maps (or jets) we define $\mathcal F_H\{t\}$ as the space of convergent power series in $t$ with coefficients in $\mathcal F$ with constant term $H$. Furthermore we denote by $\mathcal F_H[t]$ (resp.\ $\mathcal F^{\ell}_H[t]$) the space of polynomials in $t$ (resp.\ polynomials in $t$ of degree less or equal to $\ell$) with constant term $H$ and coefficients belonging to the space $\mathcal  F$. Note that $\mathcal F_H^\ell [t]\subset \mathcal F_H[t] \subset \mathcal F_H\{t\}$. 
Furthermore given an integer $\ell \in \N$ we define the truncation map $\pi_\ell: \mathcal F\{t\} \rightarrow \mathcal F^\ell[t]$ as \[\pi_\ell \left(\sum_{k \geq 0} F_k t^k\right) = \sum_{k=0}^\ell F_k t^k.\]
Moreover we define the tautological map $\tau_\ell: \mathcal F^\ell \rightarrow \mathcal F_H^\ell[t]$ as 
\[\tau_\ell(F_1,\ldots, F_\ell) = H+ \sum_{k=1}^\ell F_k t^k .\]
If the identification of $\mathcal F^\ell$ and $\mathcal F_H^\ell[t]$ is evident from the context we avoid the application of $\tau_\ell$ notationally.
\end{definition}

\begin{definition}\label{def:parcur} Let $H(t)\subset \holmaps(M, \C^{N'})$ be a smooth curve such that $H(0)\in \holmaps(M,M')$. We say that $H(t)$ is \emph{tangent to $\holmaps(M,M')$ to order $r$} at $H(0)$ if for any local 
parametrization $Z(s)$ of $M$ and any real-analytic function $\varrho'$ defined in a neighbourhood of $H(Z(0))$ in $\CNp$ and
vanishing on $M'$ we have that 
 $\rho'(H(Z(s),t))=O(t^{r+1})$. We denote the set of such parametrized curves by $\parcur{r}$ (or $\parcur{r}_H$ if we need to emphasize that $H=H(0)$).
\end{definition}

\subsection{Infinitesimal deformations of higher order}
\label{section:introhigherorderinfdef}

Assume that $M\subset \mathbb C^{N}$ and $M'\subset \mathbb C^{N'}$ are closed generic real-analytic submanifolds, and that $H: M \to \mathbb C^{N'}$ is a CR map with $H(M) \subset M'$. Denote by $\Gamma_H = \Gamma_{CR}(H^{*}(\mathbb C T(\mathbb C^{N'})))$ the space of real-analytic
CR sections of the pull back bundle of $\mathbb C T(\mathbb C^{N'})$ with respect to $H$.
Let $V \in \Gamma_H$, that is an expression of the kind
\[V =  \sum_{j=1}^{N'}V_j(Z)\frac{\partial}{\partial Z_j'}\] 
where each $V_j(Z)$ are real-analytic CR functions on $M$.

\begin{definition}\label{def:infDef}
We say that $V$ is an \emph{infinitesimal deformation of $H$} if the real part of $V$ is tangent to $M'$ along $H(M)$, that is if for every $p\in M$ and every real-analytic function $\rho'$ defined in a neighbourhood of $H(p)$ and
vanishing on $M'$ we have 
\begin{equation}\label{eq:infdef}
\real \sum_{j=1}^{N'} V_j (Z) \rho'_{Z_j'} (H(Z), \overline{H(Z)}) = 0,  \ \ \quad  \forall \ Z\in M \cap U,
\end{equation}
for some open neighbourhood $U$ of $p$.  We denote this subspace of $H^{*}(\mathbb C T(\mathbb C^{N'}))$ by $\mathfrak {hol}(H)$. (cf.\ \cite{dSLR15a, CH02})
\end{definition}

\begin{remark}
\label{rem:onepallp} If $M$ is assumed to be connected, it is enough to assume that $V$ is an infinitesimal 
deformation for $(M,p)$ for some $p\in M$, by the real-analyticity of the equation defining an infinitesimal 
deformation and standard connectivity arguments using the identity principle for real-analytic functions. 
\end{remark}

\begin{remark}
\label{rem:biholInvariance}
Let $H: M \rightarrow M'$ be a CR map and $\varphi$ and $\varphi'$ biholomorphisms defined in open 
neighbourhoods $U$ of $M$ in  $\C^N$ and $U'$ of $M'$ in $\C^{N'}$, respectively.
The space of infinitesimal deformations is biholomorphically invariant in the sense that if $\widetilde H = \varphi' \circ H \circ \varphi^{-1}$, then there is an isomorphism between $\hol(H)$ and $\hol(\widetilde H)$.
More precisely, one has the following transformation rule for infinitesimal deformations under biholomorphisms: Let $V(Z) =  \sum_{j=1}^{N'}V_j(Z)\frac{\partial}{\partial Z_j'}\in \hol(H)$, then we have
\begin{align*}
\widetilde V =  (D \varphi' V) \circ \varphi^{-1} \in \hol(\widetilde H).
\end{align*}
In particular if $H$ and $\widetilde H$ have  spaces of infinitesimal deformations of different dimensions, then one map cannot be contained in the $G$-orbit of the other.
\end{remark}

\begin{definition}\label{rem:reformulate} 
We say that $X \in \Gamma_H^k = \Gamma_H \times \dots \times \Gamma_H $  belongs to $\hol^k(H)$ if $\tau_k(X)\in \mathcal H_H[t]\cap \parcur{k}$.
In this case we call $X\in \Gamma^k_H$ an \textit{infinitesimal deformation of $H$ of order $k$}. Furthermore, for $j\leq k$ we define the projection $\pi_j:\hol^k(H)\to\hol^j(H)$, where we use the same symbol as in \cref{def:PiTau} by an abuse of notation, as 
\[\hol^k(H) \ni X=(X_1,\ldots,X_k)\to \pi_j(X)=(X_1,\ldots, X_j)\in \hol^j(H)\]
 and we put $\hol_j^k(H)=\pi_j(\hol^k(H))$.
\end{definition}

Similarly to \cref{def:infDef}, if $M$ is assumed to be connected, we can equivalently ask for 
\cref{rem:reformulate} to be satisfied only near a single point $p\in M$. This is going to be clear after we 
introduce defining equations for $\hol^k(H)$, which are very suitable for applications.

\begin{remark}	\label{rem:defPolyInfDef}
	There are universal polynomials $P_k$ depending on $(V^1, \ldots, V^{k-1})$ and their complex conjugates 
	as well as the  derivatives of $\rho'$ w.r.t. $Z'$ and $\bar Z'$ of order at most $k$ (evaluated along the map $H$), whose coefficients are combinatorial constants determined by the Faa-di-Bruno formula, such that
	\begin{align*}
	\frac{d^k}{d t^k} \rho'(h(t),\bar h(t)) = 2 \real\left(\frac{d^k h}{d t^k}  \rho'_{Z'}\right) + P_k\left(\rho'_{{Z'}^\alpha {\bar {Z'}}^\beta}  (h(t), \overline{h(t)}) , \frac{d^j h}{d t^j},  \overline{\frac{d^j h}{d t^{j}}} \colon \begin{array}{c}
		|\alpha| + |\beta| \leq k \\ j\leq k -1
	\end{array} \right).
	\end{align*}
	We will use these polynomials in order to facilitate defining equations for the space of infinitesimal deformations 
	of order $k$. To this end we drop the dependence on the derivatives of the defining function notationally.
\end{remark}
For example, in the case $k=1$ we have $P_1=0$ and for $k = 2$ we have:
\begin{align*}
P_2(V^1)= 2 \real\left( \sum_{i,j=1}^{N'}\rho'_{Z'_i Z'_j} V^1_i V^1_j + \sum_{i,j=1}^{N'}\rho'_{Z'_i \bar Z'_j} V^1_i \bar V^1_j \right).
\end{align*}

Using the polynomials defined in \cref{rem:defPolyInfDef} we can reformulate \cref{rem:reformulate} as follows:

\begin{lemma}\label{contain}
	 $V=(V^1,\ldots,V^k) \in \Gamma_H^k$ is an {infinitesimal deformation of $H$ of order $k$} if and only if $V$ satisfies the following system of equations:
	\begin{align}\label{eq:higherOrderInfDef}
	2 \real\left( \sum_{j=1}^{N'} V^{\ell}_j \rho'_{Z_j'} \right) + P_\ell(V^1,\ldots, V^{\ell-1}) = 0 \ \ {\rm for\ all} \ Z\in M\cap U, 1 \leq \ell \leq k,
	\end{align}
	for any  real-analytic function $\rho'$   defined near $H(p)$, vanishing on $M'$. The function $\rho'$ and its derivatives are computed along $H$, and  $U$ is some neighborhood of (some) $p\in M$. 
\end{lemma}
\begin{proof}
	Given $H\in \mathcal H(M,M')$ and $k \in \N$, let $X$ be as in \cref{rem:reformulate}, fix $p\in M$, and a real-analytic function $\rho'$ as in the statement of the Lemma. 
	Define $H(t) = \tau_k(X)$ and assume that $H(t) \in \parcur{k}$, that is 
	\begin{align}\label{eq:containOrder}
	\rho'(H(Z,t),\overline{H(Z,t)} )= O(t^{k+1}),
	\end{align}
	for all $Z\in M$ close to $p$, where we have written $H(t,Z)= H(t)(Z)$.
	Differentiating with respect to $t$ up to order $k$ and using \cref{rem:defPolyInfDef} 
 we can write 
  \[\rho'(H(Z,t),\overline{H(Z,t)} )= \sum_{\ell\geq 0} c_\ell t^\ell,\] 
 where 
 \[c_\ell = 2 \real\left( \sum_{j=1}^{N'} \frac{\partial^{\ell}H_j}{\partial t^{\ell}}(Z,0) \rho'_{Z_j'} \right) + P_\ell\left(\frac{\partial H}{\partial t}(Z,0),\ldots, \frac{\partial^{\ell-1}H}{\partial t^{\ell-1}}(Z,0) \right). \] 
 Hence \eqref{eq:containOrder} is equivalent to $c_\ell = 0$ for all $Z\in M$ close to $p$, $1 \leq \ell \leq k$.
	This means that $\tau_k^{-1}(H(t))$ satisfies \eqref{eq:higherOrderInfDef} and thus $X = \tau_k^{-1}(H(t))$ satisfies our system of equations.
	On the other hand let $X$ satisfy the system of equations and put $H(t) = \tau_k(X)$, then $c_{\ell} = 0$ for $1\leq \ell \leq k$. Thus $X$ satisfies  \cref{rem:reformulate}. 
	\end{proof}

\begin{definition}
	\label{def:autk}
	We define $\aut^k(H) \subset \hol^k(H)$ to be the set of $V=(V_1,\ldots, V_k)$ such that there exists a curve $g(t) \in G$ with $g(0)\cdot H = H$ satisfying $\tau_k^{-1}(\pi_k(g(t) \cdot H)) = V$.
	We will call $\aut^k(H)$ the set of \emph{trivial infinitesimal deformations of $H$ of order $k$}. 	\\ In particular if $k=1$ a small computation shows that $\aut^1(H) = \aut(H) = H_*(\hol(M))+\hol(M')|_H$ (cf. \cref{l:dimension} below). 
\end{definition}

\begin{remark}
For later reference, we list here some simple facts regarding the spaces defined above. For all $k_1,k_2,j\in \mathbb N$ with $j\leq k_1<k_2$ we have
\begin{itemize}
\item $\pi_{k_1}\circ\pi_{k_2}=\pi_{k_1}$
\item $\aut^j(H)=\pi_j(\aut^{k_2}(H))\subset  \hol^{k_2}_{j}(H)\subset \hol^{k_1}_{j}(H)\subset \hol^j(H)$
\end{itemize}
Note that the inclusion $\hol^{k_2}_{j}(H)\subset \hol^{k_1}_{j}(H)$ comes from the facts that $\pi_{k_1}(\hol^{k_2}(H))\subset \hol^{k_1}(H)$ and  $\pi_j(\hol^{k_2}(H))=\pi_{j}(\pi_{k_1}(\hol^{k_2}(H)))\subset \pi_{j}(\hol^{k_1}(H))$.

\end{remark}

The following lemma shows that the set $\hol^k_1(H)$ is invariant under the linear action of the subspace $\aut(H)$.

\begin{lemma}
	\label{lem:invarianceAut}
For all $k\geq 1$, $\hol^k_1(H)=\hol^k_1(H)+\aut(H)$.
\end{lemma}
\begin{proof}
Let $V_1\in \hol^k_1(H)$, $\widetilde V\in \aut(H)$. We need to show that $V_1 + \widetilde V\in \hol^k_1(H)$. Since $\widetilde V\in \aut(H)$ there exist one parameter families $\psi(t)\in \Aut(M)$ and $\phi(t)\in \Aut(M')$ such that $\phi(0) \circ H \circ \psi(0)^{-1}= H$ and if we write $\tilde h (t) \coloneqq  \phi(t) \circ H \circ \psi(t)^{-1}$, 
\[\widetilde V = \frac{d}{d t} \tilde h(t)|_{t=0} = \dot \phi(0) - DH \dot \psi(0) . \]

Since $V_1\in \hol^k_1(H)$, there exist $V_2,\ldots, V_k$ such that $V=(V_1,\ldots,V_k)\in \hol^k(H)$. Let $h(t)=H+V_1t+\ldots +V_k t^k$, and write 
$h(t,Z)= h(t)(Z)$ (with similar notation for all of the families appearing). By \cref{rem:reformulate} and  \cref{def:parcur}  we have that for any real-analytic function $r$ defined near $H(p)$ for some $p\in M$ and vanishing
on $M'$ that, for some neighborhood $U$ of $p$, 
\[r(h(t), \overline{h(t)})|_{M\cap U}=O(t^{k+1}).\] 
 The above equations translates 
into the fact that for any parametrization $Z(s)$ of $M$ with $Z(0) =p$ and any $r$ as above  we have that
\begin{align}
\label{eq:OrderParam}
r(h(t,Z(s)),\overline{h(t,Z(s))})= \sum_{j=k+1}^\infty \beta_{j}(s) t^j.
\end{align}

Fix any such $p$ and real-analytic function $r(Z', \bar Z')$ vanishing along $M'$ near $H(p)$, and denote by
 $\rho' = (\rho'_1, \dots , \rho'_d)^t$ a defining function of $M'$ 
 near $H(p)$.

First we claim that if we write $\hat h(t) \coloneqq \phi(t) \circ h(t) \circ \psi(t)^{-1}$, then
\begin{align}
\label{eq:OrderK}r\left( \hat h(t, Z(s)), \overline{\hat h(t, Z(s))}\right)=O(t^{k+1}).
\end{align}
Since $\phi(t) \in \Aut(M')$, there exists a family of $d\times d$-matrices $A(t)$ such that $r(\phi(t,Z'), \overline{\phi(t,Z')} ) = A(t) \rho'(Z', \bar Z')$ and since $\psi(t) \in \Aut(M)$ there exists a mapping $\theta : \R^{2n+d+1} \rightarrow \R^{2n+d}$, such that $\psi(t)^{-1}(Z(s)) = Z(\theta(t,s))$. Then the left-hand side of \eqref{eq:OrderK} becomes
\begin{align*}
	r(\hat h(t,Z(s)), \overline{\hat h(t,Z(s))}) & = 
	A(t) \rho'(h(t,\psi(t,Z(s)))) = A(t) \rho'(h(t,Z(\theta(t,s)))) \\
	& = A(t)\sum_{j=k+1}^\infty {\tilde \beta}_{j}(\theta(t,s)) t^j = O(t^{k+1}),
\end{align*}
using \eqref{eq:OrderParam} for $r=\rho'_j$, for $j=1,\dots, d'$, and writing
$\tilde \beta$ for the corresponding vector of $\beta$'s.
Let now $\hat V_j = \frac{1}{j!} \frac {d^j \hat h(t)}{d t^j}|_{t=0}$ for $1 \leq j \leq k$. Since $p$ and $r$ were arbitrary,  \eqref{eq:OrderK} implies that  $\hat V = (\hat V_1,\ldots, \hat V_k) \in \hol^k(H)$, which establishes the claim. 
 Thus $\hat V_1 \in \hol^k_1(H)$ and if we write $\tilde V = \dot \phi(0) - DH \dot\psi (0)$ we have
\begin{align*}
	\hat V_1 & =  \frac {d \hat h(t)}{d t}|_{t=0} =  \frac {d \phi(t, h(t, \psi(t,Z)^{-1})}{d t}|_{t=0}  \\
	& = \dot \phi(0) + \dot h (0) - DH \dot\psi(0). 
	 \\ & =   V_1 + \tilde V,
\end{align*}
which proves the lemma.
\end{proof}


\subsection{CR geometry}\label{sec:CRgeometry}
In this subsection we recall some standard notation from CR geometry, for more details we refer the interested reader to e.g. \cite{BERbook}.
For a generic real-analytic CR submanifold $M\subset \mathbb C^N$ we denote its CR dimension by $n$ and its real codimension by $d$  such that $N=n+d$. In this case it is well-known (cf. \cite{BERbook}) that one can choose \textit{normal coordinates}
 $(z,w)\in \mathbb C^n\times \mathbb C^d=\mathbb C^N$ such that the \textit{complexification} $\mathcal M\subset \mathbb C^{2N}$ of $M$ in coordinates $(z,\chi,w,\tau) \in \C^n \times \C^n \times \C^d \times \C^d$ is 
 given by
\begin{align*}
w = Q(z,\chi, \tau), \quad \mathrm{(or~equivalently:} \ \tau = \overline Q(\chi,z, w) {\rm )},
\end{align*}
for a suitable germ of a holomorphic map $Q:\mathbb C^{2n+d}\to\mathbb C^{d}$ additionally satisfying $Q(z,0,\tau) \equiv Q(0,\chi,\tau) \equiv \tau$ and $Q(z,\chi, \overline Q (\chi, z, w)) \equiv w$.

A basis of CR and anti-CR vector fields, which are tangent to $\mathcal M$, we denote by $L_j$ and $\bar L_j$, respectively, where $j=1,\ldots, n$.

We also recall the definition of the Segre maps. For any $j\in \N$ let $(x_1,\ldots, x_j)$ be coordinates of $\mathbb C^{nj}$, where $x_\ell\in \mathbb C^n$ for $\ell = 1, \ldots, j$. For better readability we will use the notation $x^{[j;k]} := (x_j,\ldots,x_k)$. The \textit{Segre map} of order $q\in \N$ is the map $S^q_0:\mathbb C^{nq}\to \mathbb C^N$ defined as follows:
\begin{align*}
S^1_0(x_1) := (x_1,0), \quad S^q_0\bigl(x^{[1;q]}\bigr) := \left (x_1, Q\left(x_1,\overline S^{q-1}_0\bigl(x^{[2;q]}\bigr) \right)\right),
\end{align*}
where $\overline S^{q-1}_0$ is the power series whose coefficients are conjugate to the ones of $S^{q-1}_0$ and $Q$ is a map satisfying the properties from above. 
The \textit{$q$-th Segre set} $\mathcal S^q_0\subset \mathbb C^N$ is defined as the image of the map $S^q_0$.
A CR submanifold $M \subset \C^N$ is called \textit{minimal at $p \in M$} if there is no germ of a CR submanifold $\widetilde M \subsetneq M$ of $\C^N$ through $p$ with the same CR dimension as $M$ at $p$.
The minimality criterion of Baouendi--Ebenfelt--Rothschild \cite{BER96} states that if $M$ is minimal at $0$, then $S_0^q$ is generically of full rank for sufficiently large $q$.

\section{The generalized jet parametrization property}\label{sec:jetparam}

In this section we introduce the generalized jet parametrization property which allows us to prove our local rigidity results. 
 The following definition is inspired by \cite{LMZ08}.

\begin{definition}\label{def:jetparam}
Let $M, M'$ be submanifolds of $\C^N$ and $\C^{N'}$ respectively, and let $\mathcal F\subset \holmaps(M,M')$ be an open subset. We say that $\mathcal F$ satisfies the \emph{generalized jet parametrization property of order $\tz\in \mathbb N$} if the following holds.

\

\noindent{\bfseries GJPP:}\emph{  For every $H\in \mathcal F$ there exist a neighborhood $U$ of $M$ in $\C^{N}$, a neighborhood $\mathcal N$ of $H$ in $\holmaps(U,\C^{N'})$, a finite collection of points $p_1,\ldots,p_m\in M$,
a finite collection of polynomials
$q_{k}(\Lambda)$ on $J_{p_k}^{\tz}$, open neighborhoods $\mathcal U_{k}=V_k\times U_k$ of  $\{p_k\} \times  j_{p_k}^{\tz}H$ in $\mathbb C^N \times J_{p_k}^{\tz}$, with $\{V_k\}$ forming an open cover of $M$ and $U_{k} = \{q_{k}\neq 0\}$, and holomorphic maps
$\Phi_{k} \colon \mathcal U_{k} \to \C^{N'} $, which are of the form 
\begin{align} \label{rationalJetParam}
\Phi_{k} (Z,\Lambda) =  \sum_{\alpha\in \N_0^N} \frac{p_{k}^{\alpha} (\Lambda)}{q_{k}(\Lambda)^{d^{k}_\alpha}} Z^\alpha, \quad  p_{k}^\alpha, q_{k} \in \C[\Lambda], \quad d^{k}_\alpha\in\N_0,
\end{align}
such that for every curve $\widetilde H(t)\in \parcur{r}$ satisfying $\widetilde H(t)\in \mathcal N$ for all $t$ the following holds:
\begin{itemize}
\item $j_{p_k}^{\tz} \widetilde  H(t) \in U_{k}$ for all $t$,
\item for all $1\leq k\leq m$ we have \[ \widetilde  H(Z,t)|_{V_{k}} = \Phi_{k} (Z,j_{p_k}^{\tz} \widetilde  H(t))+O(t^{r+1}).\]
\end{itemize}
}
\emph{
In particular, there exist neighborhoods $W_k$ of $j_{p_k}^{\tz}H$ in $J_{p_k}^{\tz}$ and (real) polynomials $c^{k}_i$, $i\in\N$ on $J_{p_k}^{\tz}$ such that
\begin{align}
\label{defEquationJetParam}
 A_k=j_{p_k}^{\tz} (\mathcal N \cap \mathcal F) = W_k\cap \{ \Lambda\in J_{p_k}^{\tz} \colon q_{k}(\Lambda) \neq 0, \, c^{k}_i (\Lambda, \bar \Lambda) =0 \}.
 \end{align}
Furthermore  for any $\widetilde H(t)\in \parcur{r}$ satisfying $\widetilde H(t)\in \mathcal N$ for all $t$, with  $\widetilde \Lambda(t) = j_{p_k}^{\tz} \widetilde H(t)$ we have 
\begin{align}\label{defEquationJetParamcurves}
c^{k}_i (\widetilde \Lambda(t), \bar{\widetilde \Lambda}(t)) = O(t^{r+1}), \qquad i \in \N.
\end{align}
}
\end{definition}

\begin{remark}
We shall say that an open subset $\mathcal F\subset \mathcal H (M,M')$ satisfies the generalized jet parametrization property 
if for each $H\in \mathcal F$
there exists a neighbourhood ${\mathcal U}_H$ of $H$ in $\mathcal{F}$ and 
an integer $\mathbf{t}_H$ such that ${\mathcal U}_H$ satisfies the GJPP of order $\mathbf{t}_H$.
\end{remark}

\begin{remark}
Applying the GJPP for $t=0$, we obtain the familiar reproducing property $\widetilde  H(Z)|_{V_{k}} = \Phi_{k} (Z,j_{p_k}^{\tz} \widetilde  H)$ for all $\widetilde H\in \mathcal F\cap \mathcal N$, and 
in particular $\Phi_{k}(p_k, j_{p_k}^{\tz}\widetilde H) \in M'$ for all such $\widetilde H$. 
In particular, for any curve $\widetilde H(t)\subset  \mathcal F \cap \mathcal N$ we get $\widetilde  H(Z,t)|_{V_{k}} = \Phi_{k} (Z,j_{p_k}^{\tz} \widetilde  H(t))$ (with no error term): in other words, the GJPP also holds for  $r=+\infty$.
\end{remark}

\begin{remark}\label{def:linjetparamgen}
	Let $M\subset \mathbb C^N$, $M'\subset \mathbb C^{N'}$ and $\mathcal F$ be as in  \cref{def:jetparam}. For any  $\mathbf t_1,\ell \in \mathbb N$,  if $\mathcal F$ satisfies the GJPP of order $\mathbf t_1$, then it also satisfies the
	\emph{jet parametrization property of order $\mathbf t_1$ 
	 for $\ell$-th order infinitesimal deformations}, which the reader can check by expanding in terms of powers of $t$:

\textit{
For all $H\in \mathcal F$, there exists a finite collection of points $q_1,\ldots,q_m\in M$, a neighborhood $\Omega_1$ of $H$ in $\mathcal H(M,\mathbb C^{N'})$, continuous functions $r_{k,i}:\Omega_1\times (J_{q_k}^{\mathbf t_1})^{\ell}\to \mathbb C^\ell$, $i\in\mathbb N$, $r_{k,i}=(r_{k,i}^1,\ldots,r_{k,i}^\ell)$  with $r_{k,i}^j(\widetilde H,\Lambda,\overline \Lambda)$ 
weighted homogeneous of degree $j$ polynomial in $(\Lambda_1,\ldots,\Lambda_j)$, and continuous maps $K_k=(K_k^1,\ldots, K_k^{\ell}):\Omega_1\times (J_{q_k}^{\mathbf t_1})^{\ell}\to (\mathbb C\{Z\}^{N'})^{\ell}$, where the $j$-th component of $K_k(\widetilde H,\Lambda,Z)$ depends on $\Lambda_1, \ldots, \Lambda_j$ and is a weighted homogeneous polynomial in these variables, satisfying the following statement.
	For any given $\widetilde H\in \Omega_1$, there exists a solution
	$V = (V^1,\ldots,V^{\ell})\in \Gamma^{\ell}_{\widetilde H}$
	to \eqref{eq:higherOrderInfDef} whose
	$\mathbf t_1$-jet at $q_k$ is $\Lambda$ for all $1\leq k\leq m$ if and only if 
	\begin{equation}\label{eq:coeffexpand}
	r_{k,i}(\widetilde H, \Lambda, \overline \Lambda)=0 \mbox{ for all } 1\leq k\leq m, \ i\in \mathbb N.
	\end{equation}
	In this case, for all $1\leq k\leq m$ the (unique) solution $V$ is given by $V(Z)=K_k(\widetilde H,\Lambda, Z)$ in a neighborhood of $q_k$ in $\mathbb C^N$.
}
\end{remark}

\begin{remark}\label{rem:twentytwo}
Since $(J_{p_k}^{\mathbf t_1})^{\ell}\cong (J_{p_k}^{\mathbf t_1})_H^\ell[t]$ and $(\mathbb C\{Z\}^{N'})^{\ell}\cong \mathcal H_H^\ell[t]$, then $K_k$ can be interpreted as a map $\Omega_1\times (J_{p_k}^{\mathbf t_1})_H^\ell[t]\to \mathcal H_H^\ell[t]$.
\end{remark}

\begin{remark}
Clearly, the statements in \cref{def:jetparam} and \cref{def:linjetparamgen} remain true if one adds more points to the $p_k$ (resp.\ the $q_k$). In particular we can take the union of the $p_k$ and the $q_k$, so that both statements are satisfied at the same time: we will still denote the union by $\{p_1,\ldots,p_m\}$. 
\end{remark}

We next state a couple of remarks and some simplifications to the preceding notation.

\begin{remark}\label{rem:diagram}
It is convenient to elaborate further on the GJPP. To simplify the statements, we use the following notation:
\begin{itemize}
\item $\mathbf p=(p_1,\ldots,p_m)$;
\item $ V=V_1\times \cdots \times V_m$, $\mathcal H|_V=\mathcal H(V_1,\C^{N'})\times \cdots\times \mathcal  H(V_m,\C^{N'})$;
\item $r_V:\mathcal H(U,\C^{N'})\to \mathcal H|_V$ is defined as $r_V(H)=(H|_{V_1},\ldots, H|_{V_m})$;
\item $J^{\tz}_{\p}=J_{p_1}^{\tz}\times \cdots \times J_{p_m}^{\tz}$; $\mathcal U=\mathcal U_1\times\cdots\times\mathcal U_m$ ($\mathcal U$ is a neighborhood of $\p\times j^{\tz}_\p(r_V(H))$ in $(\C^{N})^m\times J^{\tz}_\p$)
\item $j_\p^{\tz}:\mathcal H|_V\to J^{\tz}_\p$ is defined as $j_\p^{\tz}(H_1,\ldots,H_m)=(j_{p_1}^{\tz} H_1,\ldots, j_{p_m}^{\tz} H_m)$;
\item $\Phi:\mathcal U\to (\C^{N'})^m$ is defined as $\Phi=(\Phi_1,\ldots,\Phi_m)$, so that for any $\Lambda\in J^{\tz}_\p$ we have that $\Phi(\cdot,\Lambda)\in \mathcal H|_V$;
\item $A=A_1\times\cdots\times A_m$ ($A$ is an analytic subset of a neighborhood of $j_\p^{\tz}(r_V(H))$ in $J_{\p}^{\tz}$)
\item For $\Lambda \in A$ we have $\Phi(\cdot, \Lambda) \in r_V(\mathcal F)$.
\end{itemize}

With this new notation the GJPP (cf.\ also \cref{rem:twentytwo}) can be summarized in the following diagram. Fixed $\ell\in \mathbb N$,

\begin{center}

\begin{tikzcd}[column sep = 4em, row sep = 3em]
  \mathcal F_H\{t\}\ar[hook]{r} \dar["\pi_{\ell}"]&   \holmaps_H\{t\} \rar["r_V"]  \dar["\pi_\ell"] & (\holmaps |_V)_H\{t\} \rar[shift left, start anchor= east, end anchor= west,"j_{\p}^{\tz}"]  \dar["\pi_{\ell}"]& (J_{\p}^{\tz})_H\{t\} \lar[shift left, start anchor= west, end anchor= east, dashed, "\Phi" ] \dar["\pi_{\ell}"]  \\
 \hol^{\ell}(H)\ar["\tau_{\ell}"]{r} 
 &  \holmaps_H^{\ell}[t] \rar["r_V"]  & (\holmaps |_V)_H[t]\rar[shift left, start anchor=east, end anchor=west,"j_{\p}^{\tz}"]   & (J_{\p}^{\tz})_H^{\ell}[t]\lar[shift left, start anchor= west, end anchor= east, dashed, "K" ] 
\end{tikzcd}
\end{center}
For better readability we have made $\Phi$ defined on all of $(J_{\p}^{\tz})_H\{t\}$ instead of $\mathcal U$.
The diagram above is commutative if we remove the dashed arrows. Furthermore, the restriction of $j_{\p}^{\tz}$ to $r_V(\mathcal F_H\{t\})$ and $r_V(\tau_\ell (\hol^{\ell}(H)))=r_V(\holmaps_H^{\ell}[t]\cap \parcur{\ell})$ is injective, and $\Phi\circ j_{\p}^{\tz}(r_V(H(t))) = r_V(H(t))$ (respectively $K\circ j_{\p}^{\tz}(r_V(X(t)))= r_V(X(t))$ for all $H(t)\in \mathcal F_H\{t\} $ (respectively for any $X(t)=\tau_\ell(X)\in \tau_\ell (\hol^{\ell}(H))=\holmaps_H^{\ell}[t]\cap \parcur{\ell}$).
\end{remark}


Next we provide the definition of the induced group action in the jet space. 

\begin{definition}\label{def:inducedAction}
Suppose that $\mathcal F$ satisfies the jet parametrization property from \cref{def:jetparam} and let $H \in \mathcal F$. Consider $p_1, \ldots, p_m \in M$ according to \cref{def:jetparam}. For $1\leq k \leq m$ denote by $\Lambda_k = j_{p_k}^{\tz} H$, and let $V_k$ be a small neighborhood of $\Lambda_k$ in $J_{p_k}^{\tz}$. Let $W_k$ be a neighborhood of $p_k$ in $\mathbb C^N$ such that $\Phi_k(\Lambda)$ is defined over $W_k$ for all $\Lambda\in V_k$. For any $g=(\sigma,\sigma')\in G$ we have that $g\cdot \Phi_k(\Lambda)=\sigma'\circ\Phi_k(\Lambda)\circ \sigma^{-1}$ is defined on $\sigma(W_k)$. We denote by $G'_{p_k}$ the set of $g\in G$ such that $p_k\in \sigma(W_k)$; note that $G'_{p_k}$ contains a neighborhood of the identity in $G$. We define an action of $G'_{p_k}$ in $V_k$  as follows:
\begin{align*}
	g \cdot \Lambda \coloneqq j_{p_k}^{\tz} (g \cdot \Phi_k(\Lambda)),  \qquad \Lambda \in V_k.
\end{align*}
Note that we are using the fact that elements of $\Aut(M)$ and $\Aut(M')$ extend holomorphically to small neighborhoods of $M$ and $M'$ respectively, and $\Phi_k(\Lambda)$ maps the neighborhood in $M$ into the other neighborhood in $M'$ if $\Lambda$ is close enough to $\Lambda_k$. 	

We define $G'_\p$ as the intersection of the sets $G'_{p_k}$ for $1\leq k\leq m$. Of course the set $G'_\p$ still contains a neighborhood of the identity in the group $G$.
\end{definition}

We will need the following result concerning the structure of the $G'_\p$-orbit of $\Lambda$ close to $\Lambda$.

\begin{lemma}\label{lem:analyticOrbit}
Let $G_1,\ldots, G_\ell$ be the connected components of $G$, with $G_1$ being the connected component of the identity. For $k=1,\ldots, \ell$ let $G'_{\p,k}=G'_\p\cap G_k$. For any $\Lambda \in A$ and $1\leq k\leq \ell$ we have that the connected component of $G'_{\p,k} \cdot \Lambda$ containing $\Lambda$ is a real-analytic embedded submanifold $W_k$ of a neighborhood $U$ of $\Lambda$ in $J^{\tz}_\p$.
\end{lemma}	

\begin{proof}
Since a neighborhood of the identity in $G$ acts on $U$ it makes sense to define an infinitesimal action $\alpha : \mathfrak g \rightarrow TU$. Let $n_0 = \dim G$ and $e_1, \ldots, e_{n_0}$ be a set of generators of $\mathfrak g$. Then each $f_j \coloneqq \alpha(e_j) \in TU$ for $1 \leq j \leq n_0$ is a real-analytic vector field over $U$. By Sussman's and Nagano's theorem, the local orbit $W_1$ (in the sense of \cite[\S3.3]{BERbook}) of $\{f_j:  1 \leq j \leq n_0\}$ through $\Lambda$ is a real-analytic submanifold of $U$ such that $T_\Lambda W_1 = \alpha_\Lambda(\mathfrak g)$.  Then any connected component of $G'_{\p,1}\cdot \Lambda$ is contained in a local orbit; in particular there is only one such connected component containing $\Lambda$, namely $W_1$. To conclude the proof of the Lemma, we repeat the same argument by considering the pushforward of $W_1$ and $\mathfrak g$ by any choice of elements $g_2,\ldots, g_\ell$ which belong to the intersection of the stabilizer of $\Lambda$ with $G'_{\p,2},\ldots,G'_{\p,\ell}$ respectively.
\end{proof}

The general strategy in proving local rigidity results will be to work with jets of maps rather than actual maps, by using the parametrization properties given in \cref{def:jetparam} and \cref{def:linjetparamgen}. What makes the arguments work is the fact that the jet parametrization commutes with the action of the group $G$ on the space of maps and on the space of jets, as was remarked in \cite[Lemma 19]{dSLR15a} (note that the second part in the proof of the lemma can be applied in this situation as long as it makes sense to apply automorphisms to elements in the jet space; the definition of $G'_\p$ was given in such a way that the same proof works.)
An important consequence of the equivariance under the $G$-action is the following


\begin{lemma}
	\label{cor:jetcontain} Let $M \subset \C^N$ and $M' \subset \C^{N'}$ be real-analytic submanifolds for which $\mathcal F \subset \mathcal H (M,M')$ satisfies the GJPP. 
Let us fix $H \in \mathcal F$. For $F \in \mathcal F$  we have that $j_{\p}^{\tz} F \in G'_\p \cdot j_{\p}^{\tz} H$ if and only if $F \in G  \cdot H$.
\end{lemma}

\begin{proof}
By assumption there exists $g \in G'_\p$ such that $j_{\p}^{\tz} F = g \cdot j_{\p}^{\tz} H$. Lemma 19 of \cite{dSLR15a} and \cref{def:jetparam} imply that 
\[F(Z) = \Phi(Z,j_{\p}^{\tz} F) = \Phi(Z,g \cdot j_{\p}^{\tz} H) = g \cdot \Phi(Z, j_{\p}^{\tz} H) = g \cdot H(Z), \]
 for all $Z \in U$. 
\end{proof}

\section{Necessary and sufficient infinitesimal conditions for local rigidity}\label{sec:mainresults}

The higher order infinitesimal deformations introduced in \cref{rem:reformulate} can be used to provide a characterization of local rigidity. More precisely we can state the following result:

\begin{theorem}\label{th:moreprecisely}
Let $M \subset \C^N$ and $M' \subset \C^{N'}$ be real-analytic submanifolds for which $\mathcal F \subset \mathcal H (M,M')$ satisfies the GJPP, and let 
$H_0\in \mathcal F$. 
Then 
there exists a neighbourhood ${\mathcal U}_{H_0}$ and a function $j\mapsto \ell(j)$ 
such that for $H\in {\mathcal U}_{H_0}$
the following hold: 
\begin{itemize}
\item[(a)] $\bigcap_{k\geq j} \hol^k_j(H)=\hol^{\ell(j)}_j(H)$ for all $j\in \mathbb N$;
\item[(b)] Assume in addition that $G$ is a Lie group with finitely many connected components. Then $H$ is locally rigid if and only if $\aut^j(H)=\hol^{\ell(j)}_j(H)$ for all $j\in \mathbb N$.
\end{itemize}
\end{theorem}

In the following theorem we will prove at the same time \cref{th:moreprecisely} (a) and the necessary condition for local rigidity in \cref{th:moreprecisely} (b):

\begin{theorem}\label{th:necessary} Let $M \subset \C^N$ and $M' \subset \C^{N'}$ be real-analytic submanifolds for which $\mathcal F \subset \mathcal H (M,M')$ satisfies the GJPP. Let $H_0\in\mathcal F$.
Then 
there exists a neighbourhood ${\mathcal U}_{H_0}$ and a function $j\mapsto \ell(j)$ 
such that for $H\in {\mathcal U}_{H_0}$
Then for each $H_0\in \mathcal F$ 
there exists a neighbourhood ${\mathcal U}_{H_0}$ and a function $j\mapsto \ell(j)$ 
such that $H\in {\mathcal U}_{H_0}$ we have $\bigcap_{k\geq j} \hol^k_j(H)=\hol^{\ell(j)}_j(H)$. Furthermore, if $H$ is locally rigid then for all $j \in \N$ we have $\aut^j(H) = \hol^{\ell(j)}_j(H)$.	
\end{theorem}

\begin{proof}
 Let $j_0\in \mathbb N$, 
  and let $\ell=\ell(j_0)$ be the integer given by Wavrik's Theorem (see \cite{Wavrik75}) applied to the system of equations $\{c^s_i=0\}$ appearing in \eqref{defEquationJetParam}. Let $X' \in \hol^{\ell}_{j_0}(H) = \pi_{j_0} \hol^{\ell}(H)$, i.e. there exists $X \in \hol^{\ell}(H)$ such that $X' = \pi_{j_0}(X)$. Write $X=(X_1,\ldots,X_\ell)$; by \cref{rem:reformulate} we have that the curve of maps $\widetilde H(t)=H+X_1t+\ldots +X_\ell t^\ell$ belongs to $\parcur{\ell}$.

Write now $j_{\p}^{\tz}X=(\Lambda_1,\ldots,\Lambda_\ell)\in (J_{\p}^{\tz})^\ell$. Then we can define a curve $\Lambda(t)\subset J_{\p}^{\tz}$ as
\[\Lambda(t) = j_{\p}^{\tz}H + \Lambda_1 t +\ldots + \Lambda_\ell t^\ell;\]
note that $\pi_\ell (\Lambda(t))= j_{\p}^{\tz}X$ and $(j_{\p}^{\tz} \circ r_V)(\widetilde H(t))=\Lambda(t)$. Furthermore, write $\Lambda'(t) = j_{\p}^{\tz}H + \Lambda_1 t +\ldots+\Lambda_{j_0} t^{j_0}$. From \cref{def:jetparam} and \eqref{defEquationJetParamcurves} we have that $\Lambda(t)$ satisfies the system of analytic equations $\{c^s_i=0\}$ up to order $\ell$ in $t$. Applying now Wavrik's theorem and Artin's Approximation Theorem (see \cite{Artin68}), we obtain an analytic curve $\widehat \Lambda(t)\subset J_{\p}^{\tz}$ which satisfies the system $\{c^s_i=0\}$ -- that is, $\widehat \Lambda(t)\in A$ for all $t$ -- and coincides with $\Lambda'(t)$ up to order $j_0$ in $t$.

Define $\widehat H(\cdot,t)=\Phi(\cdot,\widehat \Lambda(t))$. By \cref{def:jetparam}, $\widehat H(t)\in \mathcal F$ for all $t$ and $\widehat H(0)=H$. Let $X''=\pi_{j_0}(\widehat H(t))\in \hol^{j_0}(H)$. We claim that $X''=X'$. Indeed, by using the commutativity of the diagram below, we have that $j_{\p}^{\tz}\circ r_V \circ \tau_{j_0}(X'')=j_{\p}^{\tz}\circ r_V\circ \tau_{j_0}(\pi_{j_0}\widehat H(t))= \pi_{j_0}\circ  j_{\p}^{\tz}\circ r_V(\widehat H(t))=\Lambda'(t)= j_{\p}^{\tz}\circ r_V \circ \tau_{j_0}(X') $:

\begin{center}

\begin{tikzcd}[column sep = 5em, row sep = 3em]
\widehat H(t)\in \mathcal F_H\{t\}\ar[hook]{r} \dar["\pi_{\ell}"] \ar[start anchor= 200, end anchor =100, bend right=50, "\pi_{j_0}" ]{ddd}& \mlnode{$\Phi(\widehat \Lambda(t))=$ \\ $\widehat H(t)\in  \holmaps_H\{t\}$}\rar[shift left, start anchor=east, end anchor=west,"j_{\p}^{\tz}\circ r_V"]  \dar["\pi_{\ell}"]& \widehat\Lambda(t)\in(J_{\p}^{\tz})_H\{t\} \lar[shift left, start anchor= west, end anchor= east, dashed, "\Phi" ] \dar["\pi_{\ell}"]  \\
X\in  \hol^{\ell}(H)\ar["\tau_{\ell}"]{r} \dar["\pi_{j_0}"]
 & \mlnode{$\tau_\ell(X)=$\\ $\widetilde H(t) \in \holmaps_H^{\ell}[t]$}  \rar[shift left, start anchor=east, end anchor=west,"j_{\p}^{\tz} \circ r_V"]   & \mlnode{$j_{\p}^{\tz}\widetilde H(t)=$ \\$\Lambda(t)\in (J_{\p}^{\tz})_H^{\ell}[t]$} \uar[start anchor = 120, end anchor = 200, bend left=30,"\rm Wavrik" style={anchor=north, rotate=-50,inner sep = 0.1cm}] \lar[shift left, start anchor= west, end anchor= east, dashed, "K" ] \dar["\pi_{j_0}"]  \\
 
\mlnode{$\pi_{j_0}(X)=$ \\ ${\mathbf X'}\in \hol^{\ell}_{j_0}(H)$}\dar[hook] \ar["j_{\p}^{\tz}\circ r_V\circ \tau_{j_0}"]{rr}& & \mlnode{$\pi_{j_0}(\Lambda(t))=$\\$\Lambda'(t)\in  (J_{\p}^{\tz})_H^{j_0}[t]$} \\

\mlnode{$\pi_{j_0}(\widehat H(t))=$ \\ $X''\in  \hol^{j_0}(H)$} \ar[start anchor = 80, end anchor = 215, bend right= 5,"j_{\p}^{\tz}\circ r_V\circ \tau_{j_0}"]{rru}&  &  
\end{tikzcd}

\end{center}
On the other hand, as observed in \cref{rem:diagram}, the restriction of $j_{\p}^{\tz}\circ r_V \circ \tau_{j_0}$ to $\hol^{j_0}(H)$ is injective by \cref{def:linjetparamgen}. It follows that $X'=X''$ as claimed. Thus we have
\[X'=\pi_{j_0} \widehat H(t)=\pi_{j_0}(\pi_k \widehat H(t))\]
for all $k\geq j_0$. Since $\pi_k \widehat H(t)\in \hol^k(H)$ it holds that $X'\in \pi_{j_0}\hol^k(H) = \hol^k_{j_0}(H)$. This shows that $X'\in \bigcap_{k \geq j_0} \hol^k_{j_0}(H)$, and since $X'$ is an arbitrary element of $\hol^{\ell(j_0)}_{j_0}(H)$ we conclude that $\hol^{\ell(j_0)}_{j_0}(H)\subset \bigcap_{k \geq j_0} \hol^k_{j_0}(H)$. Since the other inclusion is trivial, it follows that $\hol^{\ell(j_0)}_{j_0}(H)= \bigcap_{k \geq j_0} \hol^k_{j_0}(H)$.

We turn now to the second statement in the theorem. Suppose that there exists $j_0 \in \N$ such that $\aut^{j_0}(H) \neq \hol^{\ell(j_0)}_{j_0}(H)$, and let $X'\in \hol^{\ell(j_0)}_{j_0}(H)$ with $X' \not\in \aut^{j_0}(H)$. Let $\widehat H(t)$ be the curve of maps obtained by applying the argument above to $X'$. Since $X'\not\in \aut^{j_0}(H)$ and $\pi_{j_0}\widehat H(t) = X'$, we deduce that the curve of maps $\widehat H(t)$ does not come from the action of $G$ (cf.\ \cref{def:autk}). Thus, $H$ is not locally rigid, contradicting the assumptions.
\end{proof}

The statement about the sufficiency of the condition for local rigidity in \cref{th:moreprecisely} (b)  can be actually slightly refined:

\begin{theorem}\label{th:ifthereexists} Let $M \subset \C^N$ and $M' \subset \C^{N'}$ be real-analytic submanifolds for which $\mathcal F \subset \mathcal H (M,M')$ satisfies the GJPP. Let $H_0\in\mathcal F$ and suppose that $G$ is a Lie group with finitely many connected components.
Let ${\mathcal U}_{H_0}$ and $\ell(j)$ be  given in \cref{th:necessary}, and 
let $H\in {\mathcal U}_{H_0}$. If there exists $j_0 \in \N$ such that for all $j \geq j_0$ it holds that $\aut^j(H) =\hol^{\ell(j)}_j(H)$ then $H$ is locally rigid.
\end{theorem}

\begin{remark}
In particular if there exists $j_0 \in \N$ such that  $\hol^j(H) = \aut^j(H)$ for all $j \geq j_0$ then the sufficient condition of \cref{th:ifthereexists} is satisfied.
\end{remark}

The following fact we need in the proof of \cref{th:ifthereexists}.

\begin{lemma}\label{lem:hironaka}
Let $A\subsetneq B\subset \mathbb R^n$ be real analytic sets, and let $p\in A$ be such that $p\in \overline {B\setminus A}$. Then there exists a real analytic curve $\gamma:[-1,1]\to\mathbb R^n$ such that $\gamma(0)=p$, $\gamma\subset B$ but $\gamma\not\subset A$.
\end{lemma}
\begin{proof}
By Hironaka's resolution of singularities \cite{Hironaka64}, there exists a regular, real analytic manifold $\widetilde B$ and an analytic surjection $\pi:\widetilde B\to B$. In particular we have that $\widetilde A=\pi^{-1}(A)\subsetneq \widetilde B$ is an analytic subset of $\widetilde B$. Let $q\in \widetilde A$ such that $\pi(q)=p$; choosing a chart $\widetilde B\supset U\to V\subset \mathbb R^m$ for $\widetilde B$ around $q$, since $\widetilde A$ is a proper analytic subset of $U$ we can consider it as a proper analytic subset of $V\subset \mathbb R^m$. Let $\widetilde \gamma:[-1,1]\to V$ be the parametrization of a straight segment such that $\widetilde \gamma(0)=q$ and $\widetilde \gamma\not\subset \widetilde A$ (such a straight segment must exist since $\widetilde A\subsetneq V$). Then $\gamma=\pi\circ\widetilde \gamma$ satisfies the conditions of the Lemma.
\end{proof}

\begin{proof}[Proof of \cref{th:ifthereexists}]
We show that if $H$ is not locally rigid then for all $j_0 \in \N$ there exists $j\geq j_0$ such that $\hol^j(H)\neq \aut^j(H)$. In the following we will refer to the notation  of \cref{def:jetparam}. What we are assuming amounts to the fact that there exists a neighborhood $V$ of $\Lambda = j_{\p}^{\tz}(r_V(H))$ such that $A\cap V$ is strictly larger than $(G'_\p\cdot \Lambda) \cap V$ (see \cref{def:inducedAction}). Therefore by \cref{lem:hironaka} there exists an analytic curve $\gamma:[-1,1]\to V$ such that $\gamma(0)=\Lambda$ and the image of $\gamma$ is contained in $A$ but is not contained in $G'_\p\cdot \Lambda$.

Consider the map $\Gamma: \mathbb C^N \times [-1,1]\to \mathbb C^{N'}$ defined by $r_V\Gamma(\cdot, t)=\Phi(\cdot,\gamma(t))$. Because of the analyticity of $\Phi$, the map $\Gamma$ is also analytic and we can write $\Gamma(Z,t)=\sum_{\ell\geq 0} \Gamma_\ell(Z) t^\ell$. Note that, since $\pi_{\ell(j)}(\Gamma(Z,t)) \in \hol^{\ell(j)}(H)$, it holds that $\pi_j(\Gamma(Z,t)) = \pi_j(\pi_{\ell(j)}(\Gamma(Z,t))) \in \pi_j(\hol^{\ell(j)}(H)) = \hol^{\ell(j)}_j(H)$.


Assume by contradiction that $\pi_j(\Gamma(Z,t)) \in \aut^j(H)$ for all $j$. By \cref{def:autk} for all $j$ there exists a curve $g_j(t) \in G'_\p$ such that $\pi_j (g_j(t) \cdot H) = \pi_j(\Gamma)$. Then
\[\pi_j ( g_j(t) \cdot \Lambda) = \pi_j ( j_{\p}^{\tz}  r_V(g_j(t) \cdot H)) =  j_{\p}^{\tz}r_V\tau_j(\pi_j (g_j(t) \cdot H))=\] 
\[ =j_{\p}^{\tz}r_V\tau_j(\pi_j (\Gamma(t))) =   \pi_j  j_{\p}^{\tz} (r_V\Gamma(t))=\pi_j(\gamma(t)) \]

where we have used the commutativity of the following diagram:

\begin{center}

\begin{tikzcd}[column sep = 4em, row sep = 3em]
  \Gamma(t)\in\mathcal F_H\{t\}\ar[hook]{r} \dar[shift right=0.7cm, start anchor = south, end anchor = north,"\pi_{j}"]&   \holmaps_H\{t\} \rar["r_V"]  \dar["\pi_j"] & (\holmaps |_V)_H\{t\} \rar[shift left, start anchor=east, end anchor=west,"j_{\p}^{\tz}"]  \dar["\pi_{j}"]& \gamma(t)\in(J_{\p}^{\tz})_H\{t\} \lar[shift left, start anchor= west, end anchor= east, dashed, "\Phi" ] \dar[shift right=1cm, start anchor = south, end anchor = north,"\pi_{j}"]  \\
 \pi_j (\Gamma(t))\in\hol^{j}(H)\ar["\tau_{j}"]{r} 
 &  \holmaps_H^{j}[t] \rar["r_V"]  & (\holmaps |_V)_H[t]\rar[shift left, start anchor=east, end anchor=west,"j_{\p}^{\tz}"]   & \pi_j(\gamma(t))\in(J_{\p}^{\tz})_H^{j}[t]\lar[ shift left, start anchor= west, end anchor= east, dashed, "K" ] \\
 
   g_j(t)\cdot H\in\mathcal F_H\{t\}\ar[hook]{r} \uar[shift left=0.7cm,  start anchor = north, end anchor = south,"\pi_{j}"]&   \holmaps_H\{t\} \rar["r_V"]  \uar["\pi_j"] & (\holmaps |_V)_H\{t\} \rar[shift left, start anchor=east, end anchor=west,"j_{\p}^{\tz}"]  \uar["\pi_{j}"]& g_j(t)\cdot \Lambda\in(J_{\p}^{\tz})_H\{t\} \lar[shift left, start anchor= west, end anchor= east, dashed, "\Phi" ] \uar[shift left=1cm, start anchor = north, end anchor = south, "\pi_{j}"]  \\
\end{tikzcd}
\end{center}

This means that the curve $t \mapsto g_j(t) \cdot \Lambda\in G'_\p \cdot \Lambda$ is tangent to order $j$ to the curve $t \mapsto \gamma(t)$. With the notation of \cref{lem:analyticOrbit}, we have that for any $j\in \mathbb N$ the curve $g_j(t)$ is contained in $G'_{\p,k(j)}$ for a certain $k(j)$, and thus $g_j(t)\cdot\Lambda$ is contained in $W_{k(j)}$. Hence there exists a $\widetilde k$ such that $k(j)=\widetilde k$ for infinitely many $j\in \mathbb N$.   Consequently, for arbitrarily large $j$ such that $k(j)=\widetilde k$ the curve $t \mapsto\gamma(t)$ is tangent to order $j$ to the real-analytic submanifold $W_{\widetilde k}$, hence $\gamma\subset W_{\widetilde k}\subset G'_\p\cdot \Lambda$,
which is a contradiction.
\end{proof}



We will now examine more in detail the case of infinitesimal deformations of order $1$. This case has been studied in \cite{dSLR15a, dSLR15b} in the situation where $(M,p)$ and $(M',p')$ are germs of submanifolds and $H$ is a germ of CR map sending $p$ to $p'$. In what follows, we will provide a stronger version of Theorem 2 in \cite{dSLR15b} and extend the statement to the case of global maps of (compact) submanifolds. We will state our results only in the latter situation; however, the corresponding statements for germs of maps can be deduced with the same arguments, using the parametrization property for germs of maps rather than \cref{def:jetparam}.

\begin{theorem}
	\label{t:rigid}
	Let $M \subset \C^N$ and $M' \subset \C^{N'}$ be real-analytic submanifolds for which $\mathcal F \subset \mathcal H (M,M')$ satisfies the GJPP as in \cref{def:jetparam} and assume that $G$ is a Lie group. Let $H: M \rightarrow M'$ be a holomorphic mapping in $\mathcal F$ satisfying $\hol(H) = \aut(H)$, then $H$ is locally rigid.
\end{theorem}

Note that in \cref{t:rigid} there is no assumption about the freeness and properness of the action of $G$; furthermore both $\hol(M')$ and $\hol(M)$ are involved rather than just $\hol(M')$. The assumption taken in Theorem 2 of \cite{dSLR15b}, on the other hand, imply automatically that $H_*(\hol(M))$ is contained in $\hol(M')|_{H(M)}$.


We will need the following statement, which can be deduced from the jet parametrization property from \cref{def:linjetparamgen} in the same way as \cite[Lemma 14]{dSLR15b}.

\begin{lemma}\label{l:manifold}
	Suppose that
	$\dim \hol (H) = \ell$. 
	Then there exists a neighborhood $U$ of $\Lambda_0$ in $J_{\mathbf p}^{\tz}$ such that any submanifold $X\subset A$ with $X\cap U\neq \emptyset$ satisfies $\dim(X) \leq \ell$.
\end{lemma}

For any $\Lambda \in J_{\p}^{\tz}$ close enough to $j_{\p}^{\tz}(H)$ and $F \in \mathcal F$ with $\Lambda = j_{\p}^{\tz}F$ we define $a_{\Lambda}: G'_\p \rightarrow J_{\p}^{\tz}$  resp. $b_F: G \rightarrow \mathcal F$ as $a_{\Lambda}(g) = g \cdot \Lambda$ resp. $b_F(g) = g \cdot F$. The \emph{infinitesimal action of $G$ at $\Lambda$} resp. $F$ is the differential of $a_{\Lambda}$ resp. $b_F$ at the identity $\id \in G$,  which we denote by $\alpha_{\Lambda}: \mathfrak g \rightarrow T_{\Lambda} (J_{\p}^{\tz})$ resp. $\beta_F: \mathfrak g \rightarrow T_F \mathcal F$.

\begin{lemma}\label{l:dimension}
	Under the assumption of \cref{t:rigid}, $G'_\p \cdot \Lambda_0$ is an immersed submanifold of $A$ of dimension $\ell_0 = \dim (\aut(H))$.
\end{lemma}

\begin{proof}
The statement that $G'_\p \cdot \Lambda_0$ is an immersed submanifold comes directly from the fact that the action of $G'_\p$ is smooth, see e.g. \cite[section 2.1]{DK}. Moreover we have the following: There exists $\tilde \Lambda \in G'_\p \cdot \Lambda_0$ arbitrarily close to $\Lambda_0$ such that the differential of $\Phi$ at $\tilde \Lambda$ is injective. This can be proved along the same lines as in \cite[Lemma 23]{dSLR15a}. It follows that in a neighborhood of $\tilde \Lambda$ in $G'_\p \cdot \Lambda_0$ the map $\Phi: G'_\p \cdot \Lambda_0 \rightarrow G \cdot H$ is a regular parametrization, i.e. a smooth map of maximal rank. Hence in a neighborhood of $\Phi(\tilde \Lambda)$ the set $G \cdot H$ is a submanifold of $(\C\{Z\})^{N'}$ of the same dimension as $G'_\p \cdot \Lambda_0$.\\
By \cite[Lemma 2.1.1]{DK} we have that for any $\Lambda \in J_{\p}^{\tz}$ the dimension of $G'_\p \cdot \Lambda$ is equal to the rank of $\alpha_{\Lambda}$. Since $a_{\Lambda_0} = j_{\p}^{\tz}\circ b_H$ we have $\alpha_{\Lambda_0} = j_{\p}^{\tz}\circ \beta_H$ (note that $j_{\p}^{\tz}: (\C\{Z\})^{N'} \rightarrow J_{\p}^{\tz}$ is linear).\\
Next we show that the image of $\beta_H$ is equal to $\aut(H)$, which implies that the rank of $\beta_H$ is $\ell_0$. Let $V=(X,X') \in \mathfrak g = \hol(M) \oplus \hol(M')$ and let $\Psi(t) = (\psi^{-1}(t), \psi'(t))$ be a smooth curve in $G$ with $\Psi(0) = \id$ and $\frac{d \Psi}{d t}(0) = (X,X')$. Then
\begin{align*}
	\beta_H(V) & = \frac{d}{d t} b_H(\Psi(t))|_{t=0} = \frac{d}{d t}( \psi'(t) \circ H \circ \psi^{-1}(t) )|_{t=0}\\
	& = \left\{\frac{d \psi'(t)}{d t}(H(\psi^{-1}(t))) + \psi'_*(t)\bigl(H(\psi^{-1}(t))\bigr) \cdot H_* \cdot  \left(\frac{d \psi^{-1} }{d t}(t)\right)\right\}\Big|_{t=0}  \\
	& = X'|_H + H_*(X),
\end{align*}
which shows that the image of $\beta_H$ agrees with $\aut(H)$.\\
Let $\tilde H = \Phi(\tilde \Lambda)$. The map $j_{\p}^{\tz}: T_{\tilde H} (G \cdot H) \rightarrow T_{\tilde \Lambda} (G'_\p \cdot \Lambda_0)$ is injective, since it is the inverse of the differential of $\Phi$ at $\tilde \Lambda$.\\
Since $\aut(\tilde H) \subset T_{\tilde H} (G \cdot H)$ and $j_{\p}^{\tz}$ is injective on $T_{\tilde H} (G \cdot H)$, we have
\begin{align*}
	 \rank({\alpha_{\tilde \Lambda}}) = \rank({j_{\p}^{\tz}\circ \beta_{\tilde H}}) = \rank({\beta_{\tilde H}}) = \dim (\aut(\tilde H)).
\end{align*}
Now, we have that $\dim (\aut(\tilde H)) = \ell_0$, because the dimension of the $G$-stabilizer is constant along orbits and thus the rank of $\beta_{H}$ is equal to the rank of $\beta_{\tilde H}$. Hence the rank of $\alpha_{\tilde \Lambda}$ is equal to $\ell_0$, which shows that $\dim (G'_\p \cdot \Lambda_0) = \ell_0$.
\end{proof}

We recall some notions from real-analytic geometry. Given a semi-analytic set $S \subset \R^m$, the \emph{regular set} $S_{reg}$ of $S$ is given by the collection of all $p \in S$ for which there exists a neighborhood $U \subset \R^m$ of $p$, such that $S \cap U$ is a real-analytic submanifold of $U$ of dimension $k(p) \in \N$. The dimension $d = \dim S$ of $S$ is the maximum of all $k(p)$ for $p \in S_{reg}$. We denote by $S_{reg}^d$ the set of all points $q \in S_{reg}$ such that $k(q) = d$ and define $S_{sing} = S \setminus S_{reg}^d$ as the \emph{singular set} of $S$. By \cite[section 17]{Lojasiewicz65} and \cite[section 7]{BM88} $S_{sing}$ is a semi-analytic subset of $S$ of dimension strictly less than $d$. Consequently we deduce the following observation:

\begin{remark}\label{r:intersect}
	Let $S$ be a semi-analytic set of dimension $d$, if $X$ is a semi-analytic subset or a smooth submanifold of dimension $d$ of $S$, then $X \cap S_{reg}^d \neq \emptyset$: Indeed suppose that $X \subseteq S_{sing}$, then $d = \dim X \leq \dim S_{sing} < \dim S = d$, which is not possible.
\end{remark}

\begin{proof}[Proof of \cref{t:rigid}]
Let $H \in \mathcal F$ and denote $\Lambda_0 = j_{\p}^{\tz}(r_V(H))$. Since $\Phi$ is a $G$-equivariant homeomorphism by \cref{cor:jetcontain}, $H$ is locally rigid if and only if there exists a $\hat \Lambda \in G'_\p \cdot \Lambda_0$ and a neighborhood $V$ of $\hat \Lambda$ in $J_{\p}^{\tz}$ such that $V \cap A = V \cap (G'_\p \cdot \Lambda_0)$. \\
By \cref{l:dimension} there exists a neighborhood $U$ of $\Lambda_0$ such that the connected component $C$ of $\Lambda_0$ in $U \cap (G'_\p \cdot \Lambda_0)$ is a submanifold of dimension $\ell_0$.\\
We claim $A$ is an analytic subset of dimension $\ell_0$. Indeed, by \cref{l:manifold} we have that $\dim A \leq \ell_0$, and since $C \subseteq A$ it follows that $\ell_0 = \dim C \leq \dim A \leq \ell_0$.\\
Thus by \cref{r:intersect} $C \cap A^{\ell_0}_{reg} \neq \emptyset$. Let $\tilde \Lambda \in C \cap A^{\ell_0}_{reg}$. Since $\tilde \Lambda$ belongs to $A^{\ell_0}_{reg}$ there is a neighborhood $W$ such that $W \cap A$ is a submanifold of dimension $\ell_0$. On the other hand by shrinking $W$ we can assume that $W \cap C$ is a submanifold of dimension $\ell_0$. Moreover since $C \subset G'_\p \cdot \Lambda_0 \subset A$ we have that $W \cap C \subset W \cap A$. It follows that $W \cap C = W \cap A$ by \cite[p. 22, Prop. 2.8]{GMT86} or \cite[Prop. 7, p. 41]{Narasimhan66}. 
Then we have 
\begin{align*}
	W \cap (G'_\p \cdot \Lambda_0) \subset W \cap A = W \cap C \subset W \cap (G'_\p \cdot \Lambda_0),
\end{align*}
and hence $W \cap A = W \cap (G'_\p \cdot \Lambda_0)$, which is what we needed to prove.
\end{proof}

\begin{remark}
	In particular the proof of \cref{t:rigid} shows that for any map $H$ satisfying $\hol(H) = \aut(H)$ its orbit is a (locally closed) embedded submanifold of $\mathcal F$. 
\end{remark}

\section{A class of maps satisfying the GJPP}\label{sec:exampleGJPP}

We now turn to the definition of an important class of maps which satisfies \cref{def:jetparam}.

\begin{definition}\label{defNondeg}
Let $M'$ be a generic real-analytic submanifold.  Given a holomorphic map
 $H=(H_1,\ldots,H_{N'})\in \mathcal H (M,M')$, a point $p\in M$, a 
 defining function $\rho' =(\rho_1',\ldots, \rho'_{d'}) \in  (\mathbb C\{Z'-H(p),\zeta' - \overline{H(p)}\})^{d'}$  for $M'$ in a neighborhood of the point $q = H(p)\in M'$, and a fixed sequence $\iota = (\iota_1,\ldots,\iota_{N'})$ of 
 multiindices $\iota_m\in\N_0^n$ and $N'$-tuple of integers $\ell = (\ell^1,\ldots, \ell^{N'})$ with $1 \leq \ell^j \leq d'$, we consider the determinant
 \begin{equation} \label{folcon}
s^{\iota,\ell}_H(Z) = \det \left(\begin{array}{ccc} 
 L^{\iota_1}\rho'_{\ell^1, Z_1'}(H(Z),\overline H(\bar Z)) & \cdots & L^{\iota_1}\rho'_{\ell^1,Z_{N'}'}(H(Z),\overline H(\bar Z)) \\  \vdots & \ddots & \vdots \\
 L^{\iota_{N'}}\rho'_{\ell^{N'}, Z_1'}(H(Z),\overline H(\bar Z)) & \cdots & L^{\iota_{N'}} \rho'_{\ell^{N'}, Z_{N'}'}(H(Z),\overline H(\bar Z))\end{array}\right).
\end{equation}

We define the open set $\mathcal F_{k}(p) \subset \mathcal H(M,M')$ as the set of maps $H$  for which there exists such a sequence of multiindices $\iota = (\iota_1,\ldots,\iota_{N'})$ satisfying $k = \max_{1 \leq m \leq N'}|\iota_m|$ and $N'$-tuple of integers $\ell = (\ell^1,\ldots, \ell^{N'})$ as above  such that $s_H^{\iota,\ell} (p)\neq 0$. We define $J_{k_0}$ as the set of all pairs $(\iota,\ell)$, where  $\iota = (\iota_1,\ldots,\iota_{N'})$ is a sequence of multiindices  with $k_0 = \max_{1\leq m \leq N'} |\iota_m|$ and $\ell = (\ell^1,\ldots, \ell^{N'})$ is as above.
We will say that $H$ with $H(M) \subset M'$ is {\em $k_0(p)$-nondegenerate} at $p$ if $k_0(p) = \min \{ k\colon H \in \mathcal{F}_k(p) \}$ is a finite number, and that $H$ is {\em $k_0$-nondegenerate} if $k_0=\max \{k_0(p)\colon p\in M\}$ is a finite number. We write $\mathcal{F}_{k_0}$ for the (open) 
subset of $\mathcal{H} (M,M')$ containing all $k_0$-nondegenerate maps.
\end{definition}

\begin{theorem}\label{jetparam} 
Let $M \subset \CN$, $M'\subset \CNp$  be generic real-analytic submanifolds with $M$ compact and minimal. Fix $k_0 \in \N$ and let $\mathbf t$ be the minimum integer, such that the Segre map $S^{\mathbf t}_p$ of order $\mathbf t$ associated to $M$  is generically of full rank for all $p\in M$. Then $\mathcal F_{k_0}$ satisfies \cref{def:jetparam} with $\tz = 2\mathbf t k_0$. \end{theorem}

In the next results, we will fix $j\in J_{k_0}$ and take $p=0$, so that all the functions and neighborhoods involved will implicitly depend on these choices. 

The following proposition is a version of the implicit function theorem in our setting.

\begin{proposition}\label{prop:IFTplus}
	Let $(x,y) \in \R^n \times \R^m, F: \R^n \times \R^m \rightarrow \R^m$ and let $r\geq 1$ be an integer. Suppose there is $(x_0,y_0)\in \R^n \times \R^m$ such that $F(x_0,y_0)=0$ and $F_y(x_0,y_0)$ is invertible. Let $U$ be a sufficiently small neighborhood of $x_0$ in $\R^n$ and $g: U \rightarrow \R^m$ be a function such that $F(x,g(x))=0$ for all $x \in U$. Denote by $(x(t),y(t)) \in \R^n \times \R^m$ a curve with $(x(0),y(0))=(x_0,y_0)$ which satisfies $F(x(t),y(t)) = O(t^r)$. Then $y(t) - g(x(t)) = O(t^r)$.
\end{proposition}

\begin{proof}
We write $h(t) \coloneqq g(x(t))$, denote $z(t) \coloneqq y(t)-g(x(t))$ and $f(t)\coloneqq F(x(t),y(t)) - F(x(t),h(t))$. Note that $f(t) = F(x(t),y(t)) = O(t^r)$. Taking the first derivative at $t=0$ we get $0=f'(0) = F_y(x_0,y_0) z'(0)$, hence $z'(0)=0$. Inductively for $k < r$ we obtain 
\begin{align*}
O(t^{r-k}) = f^{(k)}(t) = F_y(x(t),y(t)) y^{(k)}(t) - F_y(x(t),h(t)) h^{(k)}(t) + ...,
\end{align*}
where the dots $\ldots$ represent a suitable polynomial in the derivatives of $F$ and $x$ of order $\leq k$ and derivatives of $y$ of order strictly less than $k$,  minus the same polynomial with $h$ in place of $y$. Putting $t=0$ we obtain that 
\begin{align*}
	0 = f^{(k)}(0) = F_y(x_0,y_0) z^{(k)}(0),
\end{align*}
where all other terms vanish by the induction assumption, thus $z^{(k)}(0)=0$.
\end{proof}

The following results follow from Proposition 25 and Corollary 26 from \cite{Lamel01} using \cref{prop:IFTplus} instead of the implicit function theorem in \cite{Lamel01}; note that the proof of those results does not depend on the assumption that $H$ sends $0$ to a fixed point (cf. \cite{dSLR15b}).

\begin{lem}\label{lem:derivBasicIdentity}
For all $\ell \in \N$ there exists a holomorphic mapping $\Psi_{\ell}: \C^N \times \C^N \times \C^{K(k_0 + \ell) N'} \to \C^{N'}$ such that for every $H(t) \in \parcur{r}$ we have
\begin{align}
\label{derivBasicIdentity}
\partial^{\ell} H(Z,t) = \Psi_{\ell}(Z,\zeta, \partial^{k_0+ \ell} \bar H(\zeta,t))+O(t^{r+1}), 
\end{align}
for $(Z,\zeta)$ in a neighborhood of $0$ in $\mathcal M$, where $\partial^{\ell}$ denotes the collection of all derivatives up to order $\ell$. Furthermore there exist polynomials $P^{\ell}_{\alpha, \beta},Q_{\ell}$ and integers $e^{\ell}_{\alpha,\beta}$ such that 
\begin{align}
\label{derivBasicIdentityRational}
\Psi_{\ell}(Z,\zeta,W) = \sum_{\alpha,\beta \in \N_0^N} \frac{P^{\ell}_{\alpha,\beta}(W)}{Q_{\ell}^{e^{\ell}_{\alpha,\beta}}(W)} Z^{\alpha} \zeta^{\beta}.
\end{align}
\end{lem}

Next, we need  as usual  to evaluate the reflection identity along the Segre sets. The proof of the following corollary is very standard: it uses \cref{lem:derivBasicIdentity} (instead of the version with $t=0$) and we refer to \cite{BER99, JL13a}.

\begin{cor}\label{cor:iterationSegre}
For fixed $q \in \N$ with $q$ even there exists a holomorphic mapping $\varphi_{q}: \C^{qn} \times \C^{K(q k_0) N'}\to \C^{N'}$ such that for every $H(t) \in \parcur{r}$, we have
\begin{align}
\label{iterationSegre}
H(S^q_0(x^{[1;q]}),t) = \varphi_{q}(x^{[1;q]}, j_0^{q k_0} H(t))+O(t^{r+1}).
\end{align}
Furthermore there exist (holomorphic) polynomials $R^{q}_{\gamma},S_{q}$ and integers $m^q_{\gamma}$ such that 
\begin{equation}
\label{iterationSegreRational}
\varphi_{q}(x^{[1;q]},\Lambda) =  \sum_{\gamma \in \N_0^{qn}} \frac{R^{q}_{\gamma}( \Lambda)}{S_{q}^{m^{q}_{\gamma}}(\Lambda)} (x^{[1;q]})^{\gamma}  
\end{equation}

\end{cor}

\begin{proof}[Proof of \cref{jetparam}]
By the choice of $\mathbf t\leq d+1$, the Segre map $S^{\mathbf t}_p$ is generically of maximal rank at any point $p\in M$. Fix any $p\in M$ and choose coordinates such that $p\cong 0$. By Lemma 4.1.3 in \cite{BER99}, the Segre map $S^{2\mathbf t}_0$ is of maximal rank at $0$. Using the constant rank theorem, there exists a neighborhood $\mathcal V$ of $S^{2\mathbf t}_0$ in $(\mathbb C\{x^{[1;2\mathbf t]}\})^N$ and a map $T:\mathcal V\to (\C\{Z\})^{2\mathbf t n}$ such that $A\circ T(A)=Id$ for all $A\in \mathcal V$.


We now define the holomorphic map
\[\phi:\mathcal V\times (\mathbb C\{x^{[1;2\mathbf t]}\})^{N'} \to (\mathbb C\{Z\})^{N'}\]
as 
\begin{equation}\label{defphi}\phi(A,\psi)=\psi(T(A)).
\end{equation}
 Thus we have  that $\phi(A,h\circ A) = h(A(T(A)))=h$ for all $A\in \mathcal V$ for all $h\in(\mathbb C\{Z\})^{N'}$. 
We define $\widetilde \Phi(\cdot, \Lambda)=\phi(S_0^{2\mathbf t},\varphi_{{2\mathbf t}}(\cdot, \Lambda))$. Now, if we apply $\phi(S_0^{2\mathbf t},\cdot)$ to both sides of equation \eqref{iterationSegre} with $q=2\mathbf t$ we get
\begin{equation}\label{reproducing}
H(t)= \phi(S_0^{2\mathbf t},H(t)\circ S^{2\mathbf t}_0)=\phi(S_0^{2\mathbf t},\varphi_{{2\mathbf t}}(\cdot, j_0^{{2\mathbf t k_0}} H(t)) + O(t^{r+1}))= \widetilde \Phi(\cdot,  j_0^{{2\mathbf t k_0}} H(t)))+O(t^{r+1}) .
\end{equation}

Recall now that \cref{cor:iterationSegre} is based on a fixed choice of $j\in J$ and $p\in M$, $p\cong 0$. In what follows, instead, we write explicitly the dependence of the objects on $j$ and $p$. By setting $q_{j,p}(\Lambda,\bar \Lambda) = S_{2\mathbf t,j,p}(\Lambda)$, where $S_{2\mathbf t,j,p}$ is given in \eqref{iterationSegreRational}, a direct computation using \eqref{iterationSegreRational} and \eqref{defphi} allows to derive the expansion in \eqref{rationalJetParam}.
Let $\mathcal U_{j,p}$ be a neighborhood of $\{0\}\times U_{j,p}$ in $\mathbb C^N$ such that $\widetilde\Phi_{j,p}$ is convergent on $\mathcal U_{j,p}$.

Consider now $\widetilde H(t)$ as in \cref{def:jetparam}, let $j_p^{2\mathbf t k_0} \widetilde H(t)$ be its $2\mathbf t k_0$-jet at $p\in M$, and for all $p\in M$ choose $j(p)\in J$ such that $(p,j_p^{2\mathbf t k_0} \widetilde H(t))\in \mathcal U_{j(p),p}$ for all $t$ close enough to $0$. Let $\Omega'_p$ be a neighborhood of $p$ in $\mathbb C^n$ and $\Omega''_p$ be a neighborhood of $j_p^{2\mathbf t k_0} \widetilde H(0)$ in $J_p^{2\mathbf t k_0}$ such that $\widetilde \Phi_{j(p),p}$ is defined on $\Omega'_p\times \Omega''_p$. Since $M$ is compact we can select finitely many points $p_1,\ldots,p_m$ such that $M\subset \cup_{1\leq k\leq m} \Omega'_{p_k}$. We define $\Phi_{k}=\widetilde \Phi_{j(p_k),p_k}$. The derivation of \eqref{defEquationJetParam} is now an application of standard arguments. In order to prove \eqref{defEquationJetParamcurves}, let $\rho', Z(s)$ be as in \cref{def:parcur}. By assumption
\[\rho'(\widetilde{H}(Z(s),t))=O(t^{r+1});\]
on the other hand, for all $1\leq k\leq m$ we have
\[\widetilde H(\cdot,t)=\Phi_k(\cdot, j_{p_k}^{2 \mathbf t k_0}\widetilde H(t))+O(t^{r+1})\]
so that
\[\rho'(\Phi_k(Z(s), j_{\p}^{2\mathbf t k_0}\widetilde H(t)))=\rho'(\widetilde H(Z(s),t)+O(t^{r+1}))=\rho'(\widetilde H(Z(s),t))+O(t^{r+1})=O(t^{r+1}).\]
Developing the previous equation as a power series in $s$, we get a collection of function $c^k_i$ (the same as in \eqref{eq:higherOrderInfDef}), such that, putting $\widetilde \Lambda(t)=j_{\p}^{2\mathbf t k_0}\widetilde H(t)$, we have $c^{k}_i (\widetilde \Lambda(t), \bar{\widetilde \Lambda}(t))=O(t^{r+1})$, as desired.  
\end{proof}

\section{Examples}\label{sec:examples}

For later reference we list all infinitesimal automorphisms of $\Sphere{2} = \{(z,w) \in \C^2: |z|^2 + |w|^2 = 1\}$, which are given as follows:
\begin{align*}
S_1 & = (\alpha  - \bar \alpha z^2 - \bar \beta z w ) \frac{\partial }{\partial z} + (\beta - \bar \alpha z w - \bar \beta w^2) \frac{\partial }{\partial w}\\
S_2 & = -\bar \gamma w \frac{\partial }{\partial z} + \gamma z \frac{\partial }{\partial w}\\
S_3 & = i s z \frac{\partial }{\partial z} + i t w \frac{\partial }{\partial w},
\end{align*}
where $\alpha, \beta,\gamma \in \C, s,t\in \R$.

\begin{definition}
	For a map $H: M \rightarrow M'$ the \textit{infinitesimal stabilizer} of $H$ is given by $(S,S') \in \hol(M) \times \hol(M')$ such that $H_*(S) = - S'|_{H(M)}$. By an abuse of notation we say that $S \in \hol (M)$ belongs to the infinitesimal stabilizer of $H$ if there exists $S' \in \hol(M')$ such that $H_*(S) = - S'|_{H(M)}$.
\end{definition}

\begin{example}
	Let $H: \Sphere{2} \rightarrow \Sphere{3} = \{(z_1,z_2,z_3) \in \C^3: |z_1|^2+|z_2|^2+|z_3|^2 = 1\}$ be given by $H(z,w) = (z^2, \sqrt{2} z w, w^2)$, which is $2$-nondegenerate in $\Sphere{2}$ and whose infinitesimal stabilizer is given by $S_2$ and $S_3$ belonging to $\hol(\Sphere{2})$. It holds that $\dim_{\R}\hol(H) = 27$ and a set of generators of a complement of $\aut(H)$ in $\hol(H)$ is given by:
	\begin{align}
	\label{eq:infDefExHomNonTrivial}
	Y = \left(\begin{array}{ccc}
	a w& - \bar a z^3  & - \bar a z^2 w \\
	- \bar b z^2 w & b z - \bar b z w^2&  0\\
	-\bar c z w^2  & -\bar c w^3 &  c z \\
	0 & d w - \bar d z^2 w & - \bar d z w^2
	\end{array} \right)
	\left(\begin{array}{c}
	\frac{\partial}{\partial z_1} \\
	\frac{\partial}{\sqrt{2}\partial z_2} \\
	\frac{\partial}{\partial z_3} 
	\end{array} \right),
	\end{align}
	where $a,b,c,d \in \C$. We would like to refer to the example given in \cite[section 8]{dSLR15a} in the case of germs at $0$ of the corresponding hypersurfaces and maps in normal coordinates, which also lists $8$ infinitesimal deformations which do not originate from trivial infinitesimal deformations and were found with a different approach.\\
	To find all elements in $\hol(H)$ given in \eqref{eq:infDefExHomNonTrivial} we proceed as follows: An infinitesimal deformation $X=(X_1,X_2,X_3) \in \hol(H)$ has to satisfy the following equation:
	\begin{align}
	\label{eq:infDefExHom}
	\real (X(z,w) \cdot \bar H(\bar z,\bar w)) = 0, \qquad (z,w) \in \Sphere{2}.
	\end{align}
	We use the techniques developed in \cite[section II]{DAngelo91} and \cite[section 5.1.4, Theorem 4]{DAngeloBook} for studying mappings of spheres.
	Consider the homogeneous expansion of $X = \sum_{k\geq 0} X^k$, where $X^k=(X_1^k,X_2^k,X_3^k)$ and each $X_j^k$ is a homogeneous polynomial of order $k$. Putting this into \eqref{eq:infDefExHom} and introducing $(z,w) \mapsto e^{i \theta}(z,w)$ for $\theta \in \R$ one obtains for $(z,w) \in \Sphere{2}$:
	\begin{align}
	\label{eq:infDefExHom2}
	\sum_{k \geq 0} (X^k(z,w) \cdot \bar H(\bar z, \bar w)) e^{i(k-2) \theta} + \sum_{\ell \geq 0} (\bar X^\ell(\bar z,\bar w) \cdot H(z, w)) e^{i(2-\ell) \theta} = 0.
	\end{align}
	It follows that the coefficients of $e^{i m \theta}$ for $m \in \Z$ of the left-hand side of the above equation have to vanish. By the reality of \eqref{eq:infDefExHom2} one needs to consider $m \geq 0$. For $m \geq 3$ one has
	\begin{align}
	\label{eq:infDefExHom3}
	X^{m+2}(z,w) \cdot \bar H(\bar z,\bar w) = 0,
	\end{align}
	for $(z,w) \in \Sphere{2}$. Since the expression on the left-hand side of \eqref{eq:infDefExHom3} is homogeneous, \eqref{eq:infDefExHom3} also holds for $(z,w) \in \C^2$, which implies that $X^{m+2} \equiv 0$ for $m \geq 3$. Next, consider $m\in \{0,1,2\}$ to obtain the following equations for $(z,w)\in\Sphere{2}$:
	\begin{align*}
	X^{4-m}(z,w) \cdot \bar H(\bar z,\bar w) + \bar X^m(\bar z,\bar w) \cdot H( z,w) = 0.
	\end{align*}
	Homogenizing these equations by multiplying each of its second expression with $(|z|^2+|w|^2)^{2-m}$ the above equations become:
	\begin{align}
	\label{eq:infDefExHom4}
	X^{4-m}(z,w) \cdot \bar H(\bar z,\bar w) + \bar X^m(\bar z,\bar w) \cdot H( z,w) (|z|^2+|w|^2)^{2-m} = 0,
	\end{align}
	for all $(z,w) \in \C^2$. Solving the system \eqref{eq:infDefExHom4} by comparing coefficients of $(\bar z,\bar w)$ and $(z,w)$ one obtains all infinitesimal deformations of $H$; the nontrivial ones are as in \eqref{eq:infDefExHomNonTrivial}.
\end{example}

With a similar approach one can compute infinitesimal deformations in the following cases:

\begin{example}
	Let $H: \Sphere{2} \rightarrow \Sphere{3}$ be given by $H(z,w) = (z^3, \sqrt{3} z w, w^3)$ which is $3$-nondegenerate in $\Sphere{2}$. Then $\hol(H) = \aut(H)$ and the infinitesimal stabilizer of $H$ is given by $S_3 \in \hol(\Sphere{2})$.
\end{example}

\begin{example}
	Let $M=\{(z,w)\in \C^2: |z|^4 + |z|^2 |w|^4 + |w|^2 = 1\}$ and $H: M \rightarrow \Sphere{3}$ be given by $H(z,w) = (z^2, z w^2, w)$ which is $3$-nondegenerate in $M$. Then $\hol(H) = \aut(H)$ and the infinitesimal stabilizer of $H$ agrees with $\hol(M) = \{S_3\}$.
\end{example}

\bibliography{ref}

\end{document}